\documentclass[reqno]{amsart}

\usepackage{accents}
\usepackage[export]{adjustbox}
\usepackage{amsmath,amssymb,amsthm}
\usepackage[english]{babel}
\usepackage{chngcntr}
\usepackage{comment}
\usepackage{graphicx}
\usepackage{hyperref}
\usepackage{indentfirst}
\usepackage[utf8]{inputenc}
\usepackage{mathtools}
\usepackage{mathrsfs}
\usepackage{stmaryrd}
\usepackage{subcaption}
\usepackage{svg}
\usepackage{tikz-cd}
\usepackage{url}
\usepackage{xcolor}
\usepackage{todonotes}

\usepackage{fullpage}

\usepackage[url=false, maxbibnames=99]{biblatex}
\newtheorem{theorem}{Theorem}[section]
\newtheorem*{theorem*}{Theorem}

\newtheorem{conjecture}[theorem]{Conjecture}
\newtheorem{corollary}[theorem]{Corollary}
\newtheorem{example}[theorem]{Example}
\newtheorem*{example*}{Example}
\newtheorem{lemma}[theorem]{Lemma}
\newtheorem{proposition}[theorem]{Proposition}

\theoremstyle{definition}
\newtheorem{definition}[theorem]{Definition}
\newtheorem{remark}[theorem]{Remark}

\numberwithin{equation}{section}
\counterwithin{figure}{section}

\DeclareMathOperator{\add}{add}
\DeclareMathOperator{\affine}{aff}
\DeclareMathOperator{\cA}{\mathcal{A}}
\DeclareMathOperator{\can}{can}
\DeclareMathOperator{\cB}{\mathcal{B}}
\DeclareMathOperator{\Center}{center}
\DeclareMathOperator{\ch}{ch}
\DeclareMathOperator{\C}{\mathbb{C}}
\DeclareMathOperator{\fd}{\mathfrak{d}}
\DeclareMathOperator{\fb}{\mathfrak{b}}
\DeclareMathOperator{\Except}{except}
\DeclareMathOperator{\In}{in}

\DeclareMathOperator{\Lift}{Lift}

\DeclareMathOperator{\Log}{Log}
\DeclareMathOperator{\Mono}{Mono}
\DeclareMathOperator{\Out}{out}
\DeclareMathOperator{\Path}{\varrho}
\DeclareMathOperator{\Pic}{Pic}
\DeclareMathOperator{\sign}{sign}
\DeclareMathOperator{\sing}{sing}
\DeclareMathOperator{\Spec}{Spec}
\DeclareMathOperator{\supp}{Supp}
\DeclareMathOperator{\R}{\mathbb{R}}
\DeclareMathOperator{\topo}{top}
\DeclareMathOperator{\toric}{toric}
\DeclareMathOperator{\vol}{vol}
\DeclareMathOperator{\D}{\mathfrak{D}}

\newcommand{\bP}{\mathbb{P}}
\newcommand{\bR}{\mathbb{R}}
\newcommand{\cPolytope}[1]{\cP_{#1}}
\newcommand{\Dual}{\ast}
\newcommand{\nDual}[1]{m_{#1}}
\newcommand{\Polytope}[1]{P_{#1}}
\newcommand{\trop}{{\mathrm{trop}}}
\renewcommand{\top}{{\mathrm{top}}}

\newcommand{\Z}{\mathbb{Z}}

\newcommand{\Proj}{\operatorname{Proj}}

\newcommand{\NE}{\operatorname{NE}}

\newcommand{\cX}{\mathcal{X}}
\newcommand{\cP}{\mathcal{P}}

\newcommand{\cD}{\mathcal{D}}
\newcommand{\cO}{\mathcal{O}}

\definecolor{Pink}{RGB}{255,0,255}
\definecolor{Blue}{RGB}{0,0,255}
\definecolor{Green}{RGB}{0,213,0}

\title{Mirror symmetric Gamma conjecture for del Pezzo surfaces}
\author{Bohan Fang}
\address{Bohan Fang, Beijing International Center for Mathematical
  Research, Peking University, 5 Yiheyuan Road, Beijing 100871, China}
\email{bohanfang@gmail.com}

\author{Junxiao Wang}
\address{Junxiao Wang, Beijing International Center for Mathematical Research, Peking University, 5 Yiheyuan Road, Beijing 100871, China}
\email{wangjunxiao@bicmr.pku.edu.cn}

\author{Yan Zhou}
\address{Yan Zhou, Department of Mathematics, Northeastern University, 43 Leon Street
Boston, MA 02115, U.S.A.}
\email{y.zhou@northeastern.edu}

\begin{document}

\begin{abstract}
    For a del Pezzo surface of degree $\geq 3$, we compute the oscillatory integral for its mirror Landau-Ginzburg model in the sense of \cite{G-H-K_I}. We explicitly construct the mirror cycle of a line bundle and show that the leading order of the integral on this cycle involves the twisted Chern character and the Gamma class. This proves a version of the Gamma conjecture for non-toric Fano surfaces with an arbitrary K-group insertion.
\end{abstract}

\maketitle

\section{Introduction} 
Throughout this paper, we assume that we work over the field $\mathbb{C}$.

\subsection{Mirror symmetry and oscillatory integrals.} Mirror symmetry relates the Gromov-Witten theory to the complex geometry of the mirror. The original mirror symmetry statement predicts genus $0$ Gromov-Witten invariants for the quintic $3$-fold from the period integral of the Calabi-Yau form \cite{C-O-G-P_1991}. When $Y$ is weakly Fano of dimension $n$, i.e. $c_1(Y)\geq 0$, the mirror is a Landau-Ginzburg model $(\mathcal X, W)$ where $\mathcal X$ is a (possibly non-compact) Calabi-Yau manifold and $W: \mathcal X\to \mathbb C$ is a holomorphic function called the \emph{superpotential}. In general, one expects the following oscillatory integral to compute genus $0$ descendant Gromov-Witten invariants
$$
\int_\Xi e^{-\frac{W}{z}} \Omega.
$$
Here $\Omega$ is the Calabi-Yau form on $\mathcal X$.

For any vector bundle $E\in K(Y)$, we define the Gamma-modified Chern character by
$$
\Psi(E)=\hat\Gamma_Y \cup (2\pi \sqrt{-1})^{\deg/2} \ch(E).
$$
Here the Gamma class for $Y$ is 
\begin{align*}
      \hat \Gamma_Y &= \prod \Gamma(1+\delta_i) \\
      &=\exp\Bigl(-\gamma c_1(Y)+\sum_{k=2}^{\infty}\zeta(k)(k-1)!\ch_k(TY)\Bigr),
\end{align*}
where $\delta_i$ are the Chern roots of the tangent bundle $T_Y$. We fix a K\"ahler class $\omega$ for $Y$, and define the quantum cohomology central charge
\begin{align*}
Z_A(E;z)=&\int_Y z^{-\mu}z^{c_1(Y)} t^{-\omega} \Psi(E) \\
+&\sum_{d\in \mathrm{Eff}(Y), k\geq 0} \frac{(-1)^k}{z^{2+k}} \left (\int_{\bar{\mathcal M}^{\mathrm{vir}}_{0,1}(Y,d)} \mathrm{ev}^*(z^{-\mu}z^{c_1(Y)} t^{-\omega} \Psi(E)) \psi^k \right ) t^{\langle \omega, d\rangle}.
\end{align*}
We should understand $t^{-\omega}$ in the integrand as a power series of $\log t$ given by $\exp(-\omega\log t)$. Here $\mu$ is the grading operator $\mu\vert_{H^{k}(Y)}=(\frac{k-n}{2})\mathrm{id}$.

On the other side, the B-model central charge is given by the oscillatory integrals
\[
Z_B(\Xi;z)=\int_\Xi e^{-\frac{W}{z}} \Omega,
\]
where $\Xi$ is a homology $n$-cycle in $H_n(\mathcal X, W^{-1}(+\infty); \mathbb Z) :=\varprojlim_{M} H_n(\mathcal X, \mathrm{Re} W > M;\mathbb Z).$ The superpotential $W$ carries a complex parameter $t'$.

Gamma classes appear in the computation of period integrals \cite{H-K-T-Y_1995,Libgober_1999}. The following conjecture equates these two central charges.
\begin{conjecture}
    There is a map
    \[
        \alpha: K(Y)\longrightarrow H_n(\mathcal X, W^{-1}(+\infty); \mathbb Z).
    \]
    Under the mirror map $\log t = \log t' + O(t')$ and when $z>0$, 
    \[ 
        Z_A(E;z) = Z_B(\alpha(E);z).
    \]
    When $Y$ is Fano the mirror map is trivial $t=t'$.
    \label{conj:ZA=ZB}
\end{conjecture}
This conjecture is proposed by Hosono \cite{Hosono_2006}, where the B-model is given by hypergeometric series to which the oscillatory integral should evaluate (see also \cite{Horja_2000, vE-vS_2006, B-H_2006,A-vS-Z_2008,Iritani_2011,Golyshev_2009, aleshkin2023higgscoulomb}). 

It is worth noting that the Gamma class appears in the definition of quantum cohomology central charge, while the B-model central charge does not contain an explicit Gamma class. Indeed the isomorphism $\alpha$ is about the integral structures: it identifies the Gamma-modified K-theoretic classes in the quantum cohomology with the Lagrangian cycles \cite{K-K-P_2008,Iritani_2009}. See \cite{aleshkin2023higgscoulomb} for recent developments in the setting of GLSMs.  

If we just consider the leading term of Conjecture \ref{conj:ZA=ZB}, the Gamma structure is already involved, and there are no Gromov-Witten invariants. The following statement is the leading order behavior of the previous conjecture by setting $z=1$.
\begin{conjecture}
    There exists an $\epsilon>0$, such that as $t\to 0^+$
    \[
        \int_{\alpha(E)} e^{-W} \Omega= \int_Y t^{-\omega} \cdot \Psi(E) + O(t^\epsilon).
    \]
    \label{conj:main}
\end{conjecture}

Iritani proved Conjecture \ref{conj:ZA=ZB} for compact weakly Fano toric orbifolds \cite{Iritani_2009} and for toric complete semi-Fano intersections with ambient classes $E$ \cite{Iritani_2011}. Abouzaid-Ganatra-Iritani-Sheridan proved Conjecture \ref{conj:main} for ambient classes $E$ in the setting of Batyrev mirror pairs of Calabi-Yau hypersurfaces using tropical geometry and SYZ duality. 

We denote $Z_B(\Xi)=Z_B(\Xi;1)$ and $Z_\top(E)=\int_Y t^{-\omega}\cdot \Psi(E)$.  Notice this leading order statement is valid under a trivial change of variables $t=t'$ even if the actual mirror map in Conjecture \ref{conj:ZA=ZB} is nontrivial, since the higher order correction from the mirror map is in $O(t^\epsilon)$.

\begin{remark}
    The map $\alpha$ is expected to be the K-theoretic version of the homological mirror symmetry functor. The mirror cycles are indeed Lagrangian submanifolds as mirror objects to coherent sheaves on $Y$. This map is shown to be consistent with known homological mirror symmetry functor for weakly Fano toric orbifolds \cite{Fang_2020} and Batyrev toric Calabi-Yau hypersurface pairs \cite{A-G-I-S_2020}.
\end{remark}

\begin{remark}
    Gamma conjectures are several related statements involving Gamma classes in the setting of quantum cohomology and/or mirror symmetry. For Fano varieties, one can express Gamma conjectures purely as asymptotic behavior of the quantum connection, and formulate Gamma I and Gamma II conjectures \cite{G-G-I_2016, G-I_2019, S-S_2020}.  Conjecture \ref{conj:main} is called a Gamma conjecture in \cite{A-G-I-S_2020}. Following \cite{iritani2023mirror}, we call Conjecture \ref{conj:main} the \emph{mirror symmetric Gamma conjecture}. 
\end{remark}

\subsection{Del Pezzo surfaces} 

The purpose of this paper is to investigate Conjecture \ref{conj:main} for del Pezzo surfaces of $\deg \geq 3$ with the Landau-Ginzburg mirror construction in the Gross-Siebert program, in particular in the setting of the construction by Gross-Hacking-Keel \cite{G-H-K-_15}. The mirror Landau-Ginzburg model counts Maslov index $2$ disks with boundary on a Lagrangian torus fiber given by the standard torus fibration \cite{Cho_2013, C-C-L-L-T_2020}. 

More generally, when $Y$ is non-toric, the Gross-Siebert program, and in particular the work of Gross-Hacking-Keel \cite{G-H-K_I} readily provides a framework for constructing the Landau-Ginzburg mirror $(\cX_t,W_t)$ to a Looijenga pair $(Y,D)$ where $D\in \vert -K_Y\vert$ is a reduced nodal curve with at least one singularity. The superpotential $W_t$ (morally) counts such Maslov index $2$ disks. 

We compute the B-model central charge using tropical geometry. By estimating the value of $W_t$ up to some small $t^\epsilon$, we compare its level sets with the level sets of its tropicalization and compute such discrepancy. The method follows the spirit of \cite{A-G-I-S_2020}. With the Landau-Ginzburg mirror model $(\cX_t,W_t)$ from \cite{G-H-K_I} (see Section \ref{sec:prelim} and \ref{sec:super} for its construction), our main theorem is the following.
\begin{theorem}[Theorem \ref{thm:main}]
Conjecture \ref{conj:main} is true for del Pezzo surfaces of $\deg \geq 3$.
\end{theorem}

\begin{example*} 
As an example, let $\bar D = \sum_{i=1}^3\bar{D}_i$ be the union of three torus-invariant lines in $\bar Y=\bP^2$. We denote $Y=\mathrm{Bl}_p \bP^2$ by the blow-up of a non-singular point $p\in \bar{D}_1$, and let $D_i$ be the proper transform of $\bar{D}_i$ and $E$ be the exceptional divisor. Given an ample class $\omega = \sum_{i=1}^3 \lambda_i D_i + cE$ and a parameter $t\in\R_{>0}$, The Landau-Ginzburg mirror $(\cX_{t},W_t)$ of $(Y,D)$ is given as 
\begin{align*}
    \cX_{t} = \Spec(\C[\vartheta_{1,t}^{\pm},\vartheta_{2,t},\vartheta_{3,t}]/(\prod_{i=1}^3\vartheta_{i,t} = t^{\sum_{i=1}^3\lambda_i}+t^{\lambda_2+\lambda_3+c}\vartheta_{1,t})),\ W_t = \sum_{i=1}^3\vartheta_{i,t}.
\end{align*}
There is a torus chart $U_{\sigma_0}\simeq(\C^*)_{z_1,z_2}^2$ in $\cX_{t}$ corresponding to the special chamber $\sigma_0$ of the scattering diagram of $(Y,D)$ such that 
\begin{align*}
    \vartheta_{1,t}|_{U_{\sigma_0}} = t^{\lambda_1}z_1, \vartheta_{2,t}|_{U_{\sigma_0}} = t^{\lambda_2}z_2, \vartheta_{3,t}|_{U_{\sigma_0}} = t^{\lambda_3}z_1^{-1}z_2^{-1}+t^{\lambda_3+c}z_2^{-1}.
\end{align*} 

We use the torus fibration $\Log_t:U_{\sigma_0} \rightarrow M_{\R},(z_1,z_2)\mapsto(\log_t|z_1|,\log_t|z_2|)$ here. For $\mathcal E = \cO$, $\alpha(\cO)$ is the positive real locus of $U_{\sigma_0}$ in $\cX_{t}$. We have that 
\begin{align*}
     Z_{\topo}(\cO) = \frac{1}{2}((\sum_{i=1}^3\lambda_i)^2-(\lambda_1-c)^2)(\log t)^2 + \gamma(2\lambda_1+3\lambda_2+3\lambda_3-3c)\log t + 4(\gamma^2+\zeta(2)) - 4\zeta(2).
\end{align*}
Let $W^\trop_t$ be the tropicalization of $W$, which could be considered as a piecewise linear approximation. The boundary of the polytope in Figure \ref{fig:example_BlptP2_tropical_polytope} is the level set $W_t^{\trop} = 0$. The orange, purple, green line segments are where $\vartheta_{1,t}^{\trop},\vartheta_{2,t}^{\trop},\vartheta_{3,t}^{\trop}$ are $0$ respectively. The coefficients before $(\log t)^2$ and $\gamma\log t$ are the volume of the polytope and the affine length of the boundary of this polytope, and $4\zeta(2)$ are contributions from the four corners.

For $E = \cO(L)$ with $L = D_2, E, D_1$, $\alpha(\cO(L)) = \alpha(\cO)\cup\alpha(\cO_L)$. Under $\Log_t$, The mirror cycle $\alpha(\cO_L)$ projects to the blue and red rays in Figure \ref{fig:example_BlptP2_mirror_cycles} (see the actual construction in Section \ref{sec:mirror_cycle}). We have that 
\begin{align*}
    Z_{\topo}(\cO(D_2)) =& Z_{\topo}(\cO) - (2\pi\sqrt{-1})\big((\sum_{i=1}^3\lambda_i)\log t + 3\gamma\big) + \frac{1}{2}(2\pi\sqrt{-1})^2,\\
    Z_{\topo}(\cO(E)) =& Z_{\topo}(\cO) - (2\pi\sqrt{-1})\big((\lambda_1-c)\log t + \gamma\big) - \frac{1}{2}(2\pi\sqrt{-1})^2,\\
    Z_{\topo}(\cO(D_1)) =& Z_{\topo}(\cO) - (2\pi\sqrt{-1})\big((\lambda_2+\lambda_3+c)\log t + 2\gamma\big).
\end{align*}
The coefficients before $(2\pi\sqrt{-1})\log t$ are the affine lengths of the blue segments inside the polytope. The coefficients before $(2\pi\sqrt{-1})\gamma$ are the numbers of the blue rays. The coefficients of $(2\pi\sqrt{-1})^2$ are contributed from the red segments and the vertices of the blue segment. Each red segment contributes to $-1/2$ and each vertex of blue segments contributes to $1/2$.
\begin{figure}
    \centering
    \includesvg[width = 0.3\textwidth]{Figures/Example_BlptP2_tropical_polytope.svg}
    \caption{$W_t^{\trop} = 0$} 
    \label{fig:example_BlptP2_tropical_polytope}
\end{figure}

\begin{figure}
    \centering
    \includesvg[width = 1.0\textwidth]{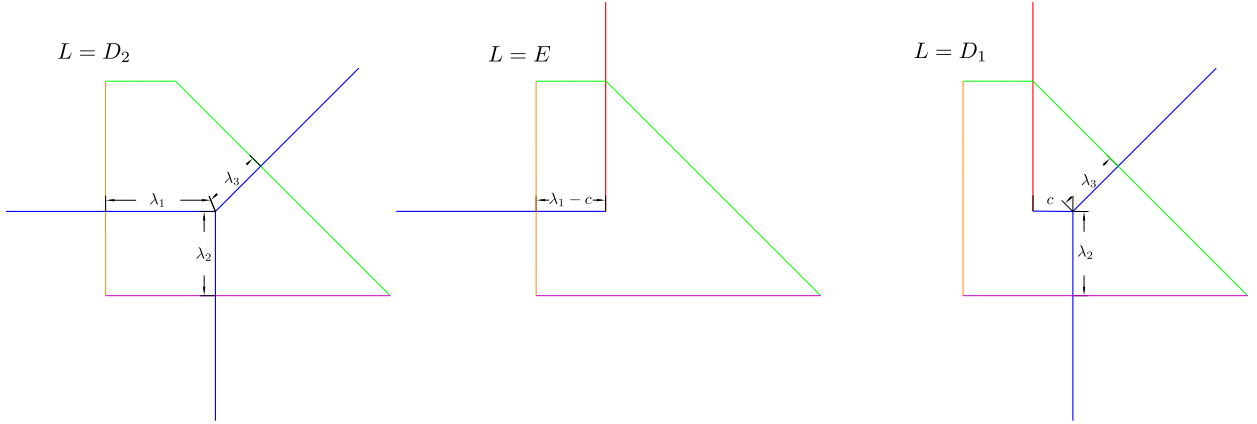}
    \caption{Mirror cycles $\alpha(\cO(L))$ for $L = D_2,E,D_1$}
    \label{fig:example_BlptP2_mirror_cycles}
\end{figure}

\end{example*}

We note that del Pezzo surfaces are complete intersections, and the mirror-symmetric Gamma conjecture for complete intersections is proved in \cite{Iritani_2011} for K-group elements pulled back from the ambient toric variety. Their superpotential is inherited from the superpotential of the ambient toric variety and we do not know whether it is the same as the intrinsic superpotential we construct here.

\subsection{Outline}

We recall the construction of the mirror Landau-Ginzburg model of a Looijenga pair $(Y,D)$ in Section \ref{sec:prelim}, and describe its superpotential in Section \ref{sec:super}. In particular, we show that there is a preferred torus chart on the mirror manifold $\cX_t$ such that the superpotential is a Laurent polynomial in Section \ref{sec:super}. In Section \ref{sec:structuresheaf} we compute the central charge $Z_B$ for the mirror cycle of the structure sheaf. Section \ref{sec:mirror_cycle} is about the construction of Lagrangian mirror cycles of any line bundle. We compute the oscillatory integrals on the mirror model for any mirror Lagrangian cycle in Section \ref{sec:anybundle}.

\subsection{Acknowledgements}

We thank Konstantin Aleshkin, Yushen Lin, Chiu-Chu Melissa Liu, Helge Ruddat, Song Yu, Eric Zalsow, Zhengyu Zong for valuable discussion. The work of BF is partially supported by NSFC 11831017, NSFC 11890661, NSFC 12125101 and a China MOST grant. The work of JW is partially supported by NSFC 11831017, NSFC 11890661 and NSFC 12125101. The work of YZ is partially supported by NKPs 2021YFA1002000, NSFC 11831017,  NSFC 11890661, and NSFC 12171006. 

\section{Preliminary}

\label{sec:prelim}

\subsection{Looijenga pairs and their mirror families}
Let us recall in this subsection some basics about Looijenga pairs and their mirror families. Our main references are \cite{G-H-K_I}, \cite{G-H-K-_15} and \cite{L-Z_22}. 
\begin{definition}
A \emph{Looijenga pair} $(Y,D)$ is a smooth rational projective surface $Y$ together with a connected reduced nodal curve $D \in |-K_Y|$ with at least one singular point. 

By the adjunction formula, the arithmetic genus of $D$ is 1. Thus, $D$ is either an irreducible rational curve with a single node or a cycle of smooth rational curves. We fix an \emph{orientation} of the cycle $D$, that is, a choice of a generator of $H_1(D,\mathbb{Z})\simeq \mathbb{Z}$ and an ordering $D=D_1 + \cdots + D_n$ of the irreducible components of $D$ compatible with the orientation. 
\end{definition}

By our assumption,  on $Y\setminus D$, there is a holomorphic volume form $\Omega$ that is unique up to scaling and has simple poles along each irreducible component of $D$. Thus, $Y\setminus D$ is a \emph{log Calabi-Yau surface}.

\begin{definition}
An \emph{internal} (-2)-\emph{curve} on a Looijenga pair $(Y,D)$ means a smooth rational curve of self-intersection -2 disjoint from $D$.  A pair $(Y,D)$ is called \textit{generic} if $Y$ has no internal $-2$ curves.
\end{definition}

\begin{definition}
An irreducible curve $E$ on $(Y,D)$ is an \emph{interior exceptional curve} if $E$ is isomorphic to a smooth rational curve, $E^2 =-1$, and $E \neq D_i$ for all $i$. Every exceptional curve on $Y$ is thus  either an interior exceptional curve or a component of $D$.
\end{definition}

\begin{definition}
Let $(Y,D)$ be a Looijenga pair. 
\begin{enumerate}
    \item A \emph{toric blow-up} of $(Y,D)$ is a birational morphism $\pi: \tilde{Y} \rightarrow Y$ such that if $\tilde{D}$ is the reduced scheme structure on $\pi^{-1}(D)$, then $(\tilde{Y},\tilde{D})$ is a Looijenga pair. In particular, $\tilde{Y}$ is smooth and $Y\setminus D = \tilde{Y} \setminus \tilde{D}$. Also, notice that if $\pi: \tilde{Y} \rightarrow Y$ is the blow-up of a node of $\bar{D}$, then $\pi$ is a toric blow-up.  
    \item A \emph{toric model} of $(Y,D)$ is a birational morphism $p:(Y,D) \rightarrow (\bar{Y},\bar{D})$ to a smooth toric surface $\bar{Y}$ with its toric boundary $\bar{D}$ such that $D \rightarrow \bar{D}$ is an isomorphism. 
\end{enumerate} 
\end{definition}

\begin{proposition}[\cite{G-H-K_I}, Proposition 1.3] Given a Looijenga pair $(Y,D)$, there exists a toric blow-up $(\tilde{Y}, \tilde{D})$ which has a toric model $(\tilde{Y}, \tilde{D}) \rightarrow (\bar{Y}, \bar{D})$.
\end{proposition}
In particular, by the above proposition, for any Looijenga pair $(Y,D)$, the corresponding log Calabi-Yau surface $U=Y\setminus D$ contains at least one open dense algebraic torus $(\C^{*})^2$. The canonical volume of $U$, up to scaling, agrees with the standard log volume form $d \log x \wedge d\log y$ of $(\C^{*})^2$ when restricted to any torus chart in $U$. 

In this paper, we fix an orientation of $D$ as follows. After a toric blow-up, we can assume that $(Y,D)$ has a toric model. Then, we orient $D=D_1 +\cdots+D_n$ by requiring that the corresponding rays of $D_1,\cdots, D_n$ in the toric fan of the toric model are ordered \emph{counterclockwise}. 

\begin{definition}
Let $(Y,D)$ be a Looijenga pair.
\begin{enumerate}\label{def:cones}
\item $\{x\in \Pic(Y)_{\R}:x^2>0\}$ is a cone with two connected components, one of which contains all ample classes. We denote this component by $C^+$.
\item For an ample class $H$, let $\tilde{\mathcal{M}} \subset \Pic(Y)$ be the set of classes $E$ such that $E^2=K_Y\cdot E=-1$ and $E\cdot H>0$. By Lemma 2.13 of \cite{G-H-K-_15}, $\tilde{\mathcal{M}}$ is actually independent of $H$. Then $C^{++}$ is defined by the inequalities $x\cdot E\ge 0$ for all $E\in \tilde{\mathcal{M}}$.
\item Let $C_{D}^{++} \subset C^{++}$ be the subcone where additionally $x \cdot \left[ D_i \right] \geq 0$ for all $i$.
\end{enumerate}   
\end{definition}

\begin{definition}
Given a Looijenga pair $(Y,D)$, we define the set of \emph{integer tropical points} $U^{\mathrm{trop}}(\mathbb{Z})$ of $U=Y \setminus D$ as follows:
\begin{align*}
     U^{\mathrm{trop}}(\mathbb{Z})& :=\{0\} \cup \{
     {\text{divisorial discrete valuations} \, \nu:\,\mathbb{C}(U)\setminus \{0\} \rightarrow \mathbb{Z} \mid \nu (\Omega) <0 } \}\\
    & := \{0\} \cup  \{(E,m) \mid m \in \mathbb{Z}_{>0}, E \subset (Z\setminus U),\,\Omega\, \text{has a pole along}\,E \}.
\end{align*}
\end{definition}
Here, in the second expression, $E$ is a divisorial irreducible component of the boundary in some partial compactification $U \subset Z$. For a toric Looijenga pair $(\bar{Y}, \bar{D})$, $U^{\mathrm{trop}}(\Z)$ can be identified with the cocharacter lattice of the algebraic torus $U$. For $U$ that is not an algebraic torus, $U^{\mathrm{trop}}(\Z)$ does not have the additive structure of a lattice.

\begin{definition}
A \emph{positive Looijenga pair} $(Y,D)$
is a Looijenga pair such that $D$ supports an ample divisor. Equivalently, $(Y,D)$ is positive if $U=Y\setminus D$ is the minimal resolution of an affine surface with at worst Du Val singularities. 
\end{definition}

\begin{definition}
Denote by $\NE(Y)$ the monoid $\NE(Y)_{\R} \cap A_1(Y,\Z)$ where $\NE(Y)_{\R} \subset A_1 (Y, \R)$ is the cone generated by effective curve classes. If the Looijenga pair $(Y,D)$ is positive, then $\NE(Y)$ will be a finitely generated monoid, and thus we can define the $\mathbb{C}$-algebra $\C[\NE(Y)]$ associated to the monoid $\NE(Y)$. 
\end{definition}

Let us recall the following main result in \cite{G-H-K_I}.
\begin{theorem}[\cite{G-H-K_I}, Theorem 1.8]
Given a positive Looijenga pair $(Y,D)$, let $R=\C[\NE(Y)]$ and 
\[
A :=\bigoplus_{q \in U^{\trop}(\Z)} R \cdot \vartheta_{q}
\]
be the free $R$-module with a basis parametrized by $U^{\trop} (\Z)$. There is a finitely generated $R$-algebra structure on $A$ such that the  structure coefficients of the algebra are determined by relative Gromov-Witten invariants of $(Y,D)$ counting rational curves meeting $D$ in a single point. Furthermore, 
\[
\cX := \Spec \, A \rightarrow \Spec \, R 
\]
is a flat affine family of Gorenstein semi-log canonical surfaces.
\end{theorem}
We often refer $\cX \rightarrow \Spec \, R$ as the \emph{mirror family} of $(Y,D)$ and $A$ as the \emph{mirror algebra}. The basis elements for $A$ in $\{\vartheta_q \mid q \in U^{\trop}(\Z)\}$ are called \emph{theta functions}.

In the later work of \cite{L-Z_22}, the moduli meaning of the mirror family is proved:
\begin{theorem}
Given a positive Looijenga pair $(Y,D)$ and its mirror family $\cX \rightarrow \Spec R$, there exists a canonical compactification of $\mathcal{X}$ to a family of surface pairs $(\bar{\cX},\cD)\rightarrow \Spec R$ such that the restriction to the open dense torus orbit $T_Y=\Pic(Y)\otimes \mathbb{G}_m \subset \Spec R$ is the universal family of marked pairs deformation equivalent to the original Looijenga pair $(Y,D)$. 
\end{theorem}
The compactified fibers in the family $\cX \rightarrow T_Y$ can be singular, with at worst Du Val singularities. We call them marked pairs because they are equipped with the auxiliary datum of a marking of the Picard group - an identification of the Picard group with $\Pic (Y)$, and a marking of the boundary - a choice of a smooth point on each component in the boundary. 

The construction of the structure coefficients for the mirror algebra $A$ has been updated with the newly developed methods to construct enumerative invariants, e.g. \cite{K-Y_19} and \cite{G-S_21}. However, the geometric interpretation of the structure coefficients will remain inside the black box throughout this paper. For our purpose of writing down the superpotential, we will use the more computationally direct method used in \cite{G-H-K_I}, that is, the machinery of \emph{scattering diagrams} and \emph{broken lines}, which we will recall in the next subsection. 

\subsection{Scattering diagrams}\label{ssec:scatt}
In this subsection, we fix a positive Looijenga pair $(Y,D)$ and assume that it has a toric model $p: (Y,D) \rightarrow (\bar{Y},\bar{D})$. Write $\bar{D}=\bar{D}_1+\cdots+\bar{D}_n$. For each $i\,(1 \leq i \leq n)$, we fix an ordered collection  $\{E_{ij}\}_{i=1}^{l_i}$ of exceptional divisors of $p$ over the component $\bar{D}_i$. Here, we only require the Looijenga pair to be positive so that $\NE(Y)$ will be finitely generated. All the results stated in this subsection will remain true for general Looijenga pairs if we choose $\sigma_P \subset A_1 (Y, \R)$ a strictly convex rational polyhedral cone containing $\NE(Y)_{\R}$ and consider instead the algebra $\C[P]$ where $P:=\sigma_P \cap A_1 (Y,\Z)$. \\

As in \cite{G-H-K_I}, $U^{\trop}(\Z)$ is the set of integer tropical points of an integral affine manifold $U^{\trop}(\mathbb\R)$ with a single singularity at the point $0$ in $U^{\trop}(\Z)$. By pushing the singularity of $U^{\trop}(\R)$ to the infinity, we get a piecewise linear isomorphism 
\[
\nu: U^{\trop}(\R) \rightarrow N_{\R}
\]
where $N$ is the cocharater lattice of the algebraic torus $\bar{Y} \setminus \bar{D}$ and $N_{\R}=N\otimes {\R}$. 

Let $M$ be the dual lattice of $N$, i.e., the character lattice of $\bar{Y}\setminus \bar{D}$. Denote by $\Sigma$ the toric fan of $\bar{Y}$ supported on $N_{\R}$. Denote by $\Sigma^{(1)}$ the set of rays in $\Sigma$ and $n_i$ the first lattice point on the ray $\rho_i$. Here, $\rho_i$ is the ray in $\Sigma$ corresponding to $\bar{D}_i$. 

\begin{remark}
Here, we use notations opposite to \cite{G-H-K_I} where $M$ will be the cocharacter lattice of the algebraic torus. We made this change to match the standard notations in the toric literature.    
\end{remark}

\begin{definition}
Let $P$ be \emph{toric monoid}, that is, a commutative monoid whose Grothendieck group $P^{\mathrm{gp}}$ is a finitely generated abelian group and $P=P^{\mathrm{gp}} \cap \sigma_P$ where $\sigma_P \subset P^{\mathrm{gp}}\otimes_{\Z}\R$ is a convex rational polyhedral cone. A \emph{$\Sigma$-piecewise linear function} $\varphi:N_{\R} \rightarrow  P^{\mathrm{gp}}_{\R}$ is a continuous function such that for each 2-dimensional cone $\sigma$ in $\Sigma$, $\varphi\mid_{\sigma}$  is given by an element $\varphi_{\sigma} \in \mathrm{Hom}(N,P^{\mathrm{gp}}) = M\otimes_{\Z}P^{\mathrm{gp}}$. 

For each $\rho \in \Sigma$ with adjacent 2-dimensional cones $\sigma, \sigma'$, we can write 
\[
\varphi_{\sigma} - \varphi_{\sigma'} = m_{\rho}\otimes \kappa_{\rho,\varphi}
\]
where $m_{\rho} \in M$ is the unique primitive element annihilating $\rho$ and positive on $\sigma$ and $\kappa_{\rho,\varphi} \in P^{\mathrm{gp}}$. We call $\kappa_{\rho,\varphi}$ the \emph{bending parameter}. If we switch the ordering of $\sigma$ and $\sigma'$, the sign of $m_{\rho}$ will flip but $\kappa_{\rho,\varphi}$ stays the same. So the bending parameters are well-defined. 

We say a $\Sigma$-piecewise linear function $\varphi: N_{\R} \rightarrow P^{\mathrm{gp}}$ is \emph{strictly $P$-convex} if each bending parameter $\kappa_{\rho,\varphi}$ of $\varphi$ is in $P \setminus P^{\times}$ where $P^{\times}$ is the group of invertible elements of $P$. 
\end{definition}

In \cite{G-H-K_I}, a wall-crossing structure called the 
\emph{canonical scattering diagram} $\mathfrak{D}^{\can}$ is drawn on $U^{\trop}(\R)$. Via the piecewise linear isomorphism $\nu$, we can define a scattering diagram $\bar{\D}=\nu(\D^{\can})$ on $N_{\R}$. It is the definition of $\bar{\D}$ that we will recall here and it is $\bar{\D}$ that we will use in this paper. \\

Our main references for this subsection are Section 3 and Section 1 of \cite{G-H-K_I}. Before we introduce the notion of scattering diagrams, we first need some basic setup. Let $E^{\oplus}$ be the monoid generated by the classes of the exceptional curves of $p$ and $E$ the lattice spanned by $E^{\oplus}$. Then, 
\begin{align*}
    A_1 (Y, \Z) &= p^{*}(A_1 (\bar{Y},\Z))\oplus E 
\end{align*}
So, we identify $A_1 (\bar{Y},\Z)$ as subgroup of $A_1 (Y, \Z)$ via $p^{*}$. 

We recall the $\Sigma$-piecewise linear function $\varphi: N \rightarrow A_1 (\bar{Y},\Z) \simeq p^{*}(A_1 (\bar{Y},\Z))$ given in Lemma 1.13 of \cite{G-H-K_I}. Define $s:\Z^{\Sigma(1)}\rightarrow N$ to send the basis $t_{\rho_{i}}$,  to $n_i$. Then, 
\[
A_1 (\bar{Y}, \Z) \ni \beta \mapsto \sum_{i=1}^{n} (\bar{D}_i \cdot \beta) t_{\rho_i}
\]
identifies $A_1(\bar{Y},\Z)$ with $\mathrm{Ker}(s)$, giving rise to an exact sequence 
\[
0 \rightarrow A_1(\bar{Y},\Z) \rightarrow \Z^{\Sigma(1)} \overset{s}{\rightarrow} N \rightarrow 0.  
\]
Let $\tilde{\varphi}:N \rightarrow \Z^{\Sigma(1)}$ be the unique piecewise linear function satisfying $\tilde{\varphi}(n_i) = t_{\rho_i}$. Let $j:\Z^{\Sigma(1)} \rightarrow A_1(\bar{Y},\Z)$ be any splitting and set $\varphi:=j\circ \tilde{\varphi}$. Then, Lemma 1.13 of \cite{G-H-K_I} shows that up to a linear function, $\varphi$ is the unique strictly $\NE(\bar{Y})$-convex $\Sigma$-piecewise linear $A_1(\bar{Y},\Z)$-valued function with bending parameters 
\[
\kappa_{\rho_i, \varphi} = [p^{*}(\bar{D}_i)].
\]
Here, we again identify $\bar{D}_i$ with its image $p^{*}(\bar{D}_i)$ in $\NE(Y) \subset A_1 (Y,\Z)$. 

Now, let 
\[
Q = \{(n,s) \in N \oplus A_1(Y,\Z) \mid \exists \,s' \in p^{*}(\NE(\bar{Y}))\oplus E \,\text{such that}\, s=s'+\varphi(n)\}. 
\]
Let $r:Q \rightarrow N$ be the projection map. Notice that by the strict $\NE(\bar{Y})$-convexity of $\varphi$, $Q^{\times}=E$. Let $\mathfrak{m}_Q = Q\setminus Q^{\times}$ and let $\widehat{\C[Q]}$ denote the completion of $\C[Q]$  with respect to the monomial ideal $\mathfrak{m}_Q$.

\begin{definition} \label{def:scatt}
We define a \emph{scattering diagram} for the pair $Q,\,r:Q \rightarrow N$. This is  a collection $\D= \{ (\mathfrak{d},f_{\mathfrak{d}}) \}$ of \emph{walls}
where 
\begin{enumerate}
    \item $\mathfrak{d} \subset N_{\R}$ is given by  
    \[
    \mathfrak{d} = \R_{\geq 0}\cdot n_0
    \]
    if $\mathfrak{d}$ is a \emph{incoming ray} and by 
    \[
    \mathfrak{d} = - \R_{\geq 0}\cdot n_0
    \]
    if $\mathfrak{d}$ is a \emph{outgoing ray} and by
    \[
     \mathfrak{d} = \R\cdot n_0 
    \]
    if $\mathfrak{d}$ is a \emph{line}, a some $n_0 \in N \setminus\{0\}$. The set $\mathfrak{d}$ is called the \emph{support} of the line or ray.  We call $(\mathfrak{d},f_{\mathfrak{d}})$ a \emph{wall} in $\D$. 
    \item $f_{\mathfrak{d}} \in \widehat{\C[Q]}$.
    \item $f \equiv 1 \mod \mathfrak{m}_Q$.
    \item $f_{\mathfrak{d}} = 1 + \sum_p c_p z^p$ for $c_p \in \C$, $r(p) \neq 0$ a positive multiple of $n_0$.
    \item for any $k >0$, there are only a finite number of rays $(\mathfrak{d}, f_{\mathfrak{d}}) \in \D$  with $f_{\mathfrak{d}} \not\equiv 1 \mod \mathfrak{m}_{Q}^{k}$.
\end{enumerate}
\end{definition}

\begin{definition}
Given a loop $\gamma$ in $N_{\R}$ around the origin, we define the \emph{path ordered product}   
\[
\theta_{\gamma, \D}: \widehat{\C[Q]} \rightarrow \widehat{\C[Q]}
\]
as follows. For each $k>0$, let $\D[k] \subset \,    \D$ be the subset  of rays with $f_{\mathfrak{d}} \not\equiv 1 \mod \mathfrak{m}_{Q}^k$. This is a set finite by (5) in Definition  \ref{def:scatt}. For $\mathfrak{d} \in \D [k]$ with $\gamma(t_0) \in \mathfrak{d}$, define 
\[
\theta^{k}_{\gamma, \mathfrak{d}}: \widehat{\C[Q]}/ \mathfrak{m}_{Q}^k \rightarrow \widehat{\C[Q]}/ \mathfrak{m}_{Q}^k
\]
by  
\[
\theta^{k}_{\gamma,\mathfrak{d}}(z^q) = z^{q} f_{\mathfrak{d}}^{\langle m_{\mathfrak{d}},r(q)\rangle}
\]
for $m_{\mathfrak{d}}$ a primitive element in $M$ satisfying, with $n$ a non-zero tangent vector of $\mathfrak{d}$, 
\[
\langle m_{\mathfrak{d}},n \rangle=0,\quad \langle m_{\mathfrak{d}},\gamma'(t_0) \rangle <0. 
\]
If $\gamma$ crosses the rays $\mathfrak{d}_1, \cdots, \mathfrak{d}_l$ in order with $\D = \{\mathfrak{d}_i\}_{i=1}^l$, we define 
\[
\theta^{k}_{\gamma, \D} = \theta^{k}_{\gamma,\mathfrak{d}_l}\circ \cdots \circ \theta^{k}_{\gamma,\mathfrak{d}_1}.
\]
We then define $\theta_{\gamma, \D}$ by taking the limit as $k \rightarrow \infty$.
\end{definition}

Here is a theorem that follows from the result obtained by Kontsevich and Soibelman in \cite{K-S_06} and proved as Theorem 1.4 in \cite{GPS}.

\begin{theorem}\label{thm:cons-scatter}
Let $\D$ be a scattering diagram. Then, there is another scattering diagram $\mathrm{Scatter}(\D)$ containing $\D$  such that $\mathrm{Scatter}(\D) \setminus \D$ consists of only outgoing rays and $\theta_{\gamma,\mathrm{Scatter}(\D)}$ is the identity for $\gamma$  a loop around the origin. 
\end{theorem}

We consider the scattering diagram $\bar{\D}_0$, over $\widehat{\C[Q]}$ given by 
\begin{equation}
    \bar{\D}_0 =
\{(\R \cdot n_i, \prod^{l_i}_{j=1}(1+ z^{(n_i,\varphi(n_i) - [E_{ij}])}) \mid 1 \leq i \leq n \}
\end{equation}
\begin{theorem}[\cite{G-H-K_I}, Theorem 3.25] \label{thm:conn-GPS}
Let $\bar{\D}$ be $\mathrm{Scatter}(\bar{\D}_0)$ in the sense of Theorem \ref{thm:cons-scatter}. In particular, $\bar{\D} \setminus {\bar{\D}_0}$ has only outgoing rays and no incoming rays. We have $\bar{\D} = \nu(\D^{\can})$.
\end{theorem}
\begin{remark}
In Subsection 3.3 of \cite{G-H-K_I}, each incoming wall in $\bar{\D}_0$ is supported on $R_{\geq 0}\cdot n_i$ instead. However, the discrepancy here is artificial. Up to the splitting of outgoing rays, there is an outgoing ray supported on $R_{\geq 0} \cdot  (-n_i)$ with the same wall-crossing function $f_{\mathfrak{d}_i}$ attached. What we did here is simply merge the two walls with the same wall-crossing function attached into a single line as one incoming wall instead. 
\end{remark}

Moreover, it is proved in Subsection 3.4 of \cite{G-H-K_I} that $\bar{\D}$ agrees with the scattering diagrams constructed  in \cite{GPS} for Looijenga pairs. Because of Theorem \ref{thm:conn-GPS}, we will use $\bar{\D}$ in this paper since it adapts to our purpose more immediately and produces the same mirror algebra as $\D^{\can}$.

More precisely, we will use a perturbed version of $\bar{\D}$. For each initial wall
\[
(\mathfrak{d}_i, f_{\mathfrak{d}_i})=(\R \cdot n_i, \prod^{l_i}_{j=1}(1+ z^{(n_i,\varphi(n_i) - [E_{ij}])})
\]
in $\bar{\D}_0$, we perturb it into $l_i$ parallel disjoint lines $\{\mathfrak{d}_{ij}\}_{j=1}^{l_i}$ to the right of the original initial wall. Here, to the right means these parallel lines are on the right of the ray $\R_{\geq 0}\cdot n_i$. These lines will have the same slope as $\mathfrak
{d}_i$ but will no longer pass through the origin. Moreover, on $\mathfrak{d}_{ij}$, the attached wall-crossing function is $f_{\mathfrak{d}_{ij}} =  z^{(n_i,\varphi(n_i) - [E_{ij}])}$. So, we can view $\{ (\mathfrak{d}_{ij}, f_{\mathfrak{d}_{ij}} \}_{j=1}^{l_i}$ as a factorization of the initial wall $(\mathfrak{d}_i, f_{\mathfrak{d}_i})$. Denote by $\bar{\D}'_0$ the collection $\{ (\mathfrak{d}_{ij}, f_{\mathfrak{d}_{ij}})\}_{i,j}$ After perturbing the initial walls in $\bar{\D}_0$ in this fashion, we can get a consistent scattering diagram $\bar{\D}' \supset \bar{\D}'_0$ by adding outgoing rays. Here, consistency means for each path that does not pass through the boundary of a wall or the intersection of two walls, the associated path-ordered product is the identity. We can interpret $\bar{\D}$ as the scattering diagram we get if we view $\bar{\D}'$ `from a great distance'. As a result, $\bar{\D}$ is also called the \emph{asymptotic scattering diagram} in \cite{GPS}. For more details, see Subsection 1.4 of \cite{GPS}.

The reason we want to use the perturbed version of $\bar{\D}$ is that in general, $\bar{\D}$ does not contain a chamber. For example, the scattering diagram for the cubic surface with $\mathbb{P}^2$ together with the toric boundary as the toric model has a ray for every rational slope. The advantage of the perturbed scattering diagram $\bar{\D}'$ is that it always contains a \emph{special chamber} (possibly unbounded) $\sigma_0$ containing the origin $0 \in N_{\R}$. This chamber corresponds to a toric chart on each smooth fiber of the mirror family $\cX \mid_{\Pic(Y)} \rightarrow T_{\Pic(Y)}$. And the expansion of the superpotential in this chamber will be Laurent polynomials, as we will see later. This suits our computational purpose.

\subsection{Broken lines}
We briefly recall the definition of broken lines and theta functions, adapted to the perturbed scattering diagram $\bar{\D}'$.
\begin{definition}
A broken line $\gamma$ on $N_{\R}$ for $q \in N\setminus\{0\}$ with a general endpoint $S$ that is not contained in the support of any wall in $\bar{\D}'$ is a proper continuous piecewise integral affine map $\gamma: (-\infty,0] \rightarrow N_{\R}$ with only a finite number of domains of linearity, together with, for each $L \subset (-\infty,0]$ a maximal connected domain of linearity of $\gamma$, a choice of monomial $n_L = c_L z^{q_L}$ where $c_L \in \mathbb{C}$ and $q_L \in Q$ satisfying the following properties:
\begin{enumerate}
    \item For the unique unbounded domain of linearity $L$, $n_L = z^{\varphi(q)}$. 
    \item For each $L$ and $t \in L$, $-r(q_L) = \gamma'(t)$ where $r: Q\rightarrow N$ is the projection map. Also, $\gamma(0) = S \in N_{\R}$.
    \item  Let $t \in (-\infty, 0)$ be a point at which $\gamma$  is not linear, passing from domain of linearity $L$ to $L'$. Let $\mathfrak{d}_1, \cdots, \mathfrak{d}_p \in \bar{\D}'$ be the walls of $\bar{\D}'$ that contain $\gamma(t)$, with attached function $f_{\mathfrak{d}_j}$. Then, we require that $\gamma$ passes from one side of walls to the other at time $t$, so that $\theta_{\gamma,\mathfrak{d}_j}$ is defined. Let $m=m_{\mathfrak{d}_j}$ be the primitive element of $M$ used to define $\theta_{\gamma, \mathfrak{d}_j}$. Expand
\[
\prod^{p}_{j=1} f_{\mathfrak{d}_j}^{\langle m, r(q_L) \rangle}
\]
as formal power series in $\widehat{\C[Q]}$. Then, there is a term $cz^s$ in this sum with 
\[
n_{L'} = n_{L} \cdot (cz^s).
\]
\end{enumerate}
\end{definition}

\begin{definition}
Given $q \in N\setminus\{0\}$ and any generic point $S \in N_{\R}$ not contained in the support of any wall in $\bar{\D}'$, we define the \emph{expansion of the theta function at $Q$, $\vartheta_{q,S}$ }, to be 
\[
\vartheta_{q,S} := \sum_{\gamma} \mathrm{Mono}(\gamma)
\]
where the sum is over all broken lines $\gamma$ for $q$ with endpoint $S$ and $\mathrm{Mono}(\gamma)$ denotes the monomial attached to the last domain of linearity of $\gamma$. 
For $0\in N$, $\vartheta_0 =1$ at any generic point $S \in N_{\R}$.
\end{definition}

\begin{remark}
For $q\in N\setminus\{0\}$,  the expansion of the theta function $\vartheta_q$ at a generic point $S$ depends on $S$. However, the algebraic relations (in the positive cases) of these theta functions are globally defined, i.e., independent of $S$. Thus, we have the well-defined mirror algebra $A$ in the positive cases. 
\end{remark}

\subsection{The action of the relative torus on the mirror family}
\begin{definition}\label{def:relative-torus}
    Let $\mathbb{A}^{D} = \mathbb{A}^{n}$ be the affine space with one coordinate for each component $D_i$ of $D$. The relative torus $T^D$ is the diagonal torus acting on $\mathbb{A}^{n}$ whose character group is the free module with basis $e_{D_1},...,e_{D_n}$. 

\end{definition}

\begin{definition}\label{def:weight-map}
There are two types of weights that will come into play associated with the action of $T^{D}$ on the mirror algebra $A$. One comes from the map $w:A_1(Y)\rightarrow \chi(T^D)$ given by
\[ C\mapsto \sum_i (C\cdot D_i) e_{D_i}\]
and the other comes from $w:N_{\R}\rightarrow \chi(T^D)\otimes \R$, where $w$ is the unique piecewise linear map with $w(0)=0$ and $w(n_i)=e_{D_i}$ for $n_i$ the primitive generator of $\rho_{i}$. Together, we get a weight map 
\begin{align*}
w: N \times \mathrm{NE}(Y) & \rightarrow \chi(T^{D}), \\
(q,C) & \mapsto w(q) + w(C).
\end{align*}
\end{definition}

The following theorem is a special case of Theorem 5.2 in \cite{G-H-K_I} when $(Y,D)$ is positive:
\begin{theorem}[Theorem 5.2, \cite{G-H-K_I}] \label{thm:rel-equi}The relative torus $T^{D}$ acts equivariantly on the mirror family $\mathcal{X}=\mathrm {Spec}(A) \rightarrow \mathrm{Spec}(\mathbb{C}[\mathrm{NE}(Y)])$. Furthermore, each theta function $\theta_{q}$, $q \in N$, is an eigenfunction of the action of $T^{D}$ with character $w(q)$. 
\end{theorem}

\section{Superpotentials and their tropicalization}

\label{sec:super}

Now, for the rest of this paper, we will  restrict to Looijenga pairs $(Y,D)$ such that $Y$ is a \emph{del Pezzo surface} of degree $  
\geq 3$. Moreover, we assume that $(Y,D)$ has a toric model $p:(Y,D) \rightarrow (\bar{Y}, \bar{D})$.

\subsection{The superpotential}

\label{sec:super_superpotential}

Denote by $\cA(Y) \subset H^2(Y,\R)$ the ample cone. 

Let $\omega \in H_2(Y,\R)$ be an element of the form $\omega = \omega_b + \omega_e$ such that 
\begin{align}
    & \omega_b = \sum_i \lambda_i D_i,\quad \lambda_i >0  \,\mathrm{for \, each}\, i; \label{eq:ample-con-1}\\ 
    &\omega_b  \,\text{is ample }; \label{eq:ample-con-2}\\ 
    &\omega_e = \sum_{i,j}c_{ij}E_{ij}, \quad c_{ij} \neq c_{ij'}  \,\text{for each $i$ and $j\neq j'$}; \label{eq:ample-con-3}\\
      & 0  < c_{ij} < \lambda_i \,\text{for each } i \,\text{and}\, j; \label{eq:ample-con-4}\\
      & \omega  \cdot D_i > \lambda_{i+1} - \min \{c_{i+1,1}, \cdots c_{i+1, l_{i+1}}. \} \label{eq:ample-con-5}
\end{align}
\begin{remark}
The conditions \ref{eq:ample-con-3}-\ref{eq:ample-con-5} are vacuous for toric del Pezzo surfaces.   
\end{remark}

\begin{lemma}
$\omega$ is contained in the ample cone $\mathcal{A}(Y)$.
\end{lemma}
\begin{proof}
First, by the conditions \ref{eq:ample-con-1} and \ref{eq:ample-con-4}, $\omega$ is in $C^{+}$. Second, since $Y$ is a del Pezzo surface, $\tilde{\mathcal{M}}$ is the set of exceptional curves on $Y$. Then, a curve class in $\tilde{\mathcal{M}}$ is either a boundary component of $D$, or one of the $[E_{ij}]$'s, or of the form $p^{*}(\beta) - \sum_{i,j} p_{ij}E_{ij} (p_{ij}\geq 0)$ for each $i$ and $j$ where $\beta$ is an effective toric divisor on $\bar{Y}$. In either case, the conditions \ref{eq:ample-con-2}, \ref{eq:ample-con-3} and \ref{eq:ample-con-4} guarantee that $\omega$ is contained in the interior of $C^{++}$. 

Furthermore, by the conditions \ref{eq:ample-con-2} and \ref{eq:ample-con-3}, we know that $\omega \cdot D_i>0$ for each $i$. Thus, $\omega$ is contained in the interior of $C^{++}_{D}$.  Finally, since $Y$ is del Pezzo and in particular contains no internal (-2)-curves, $\omega$ is in $\mathcal{A}(Y)$ by Lemma 2.15 of \cite{G-H-K-_15}.  
\end{proof}

For the rest of this paper, we will fix an element $\omega$ satisfying the conditions \ref{eq:ample-con-1} - \ref{eq:ample-con-5}. In particular, $\omega$ is ample by the above lemma. Later, we will see that the conditions \ref{eq:ample-con-3} - \ref{eq:ample-con-5} guarantee that the tropicalization of the superpotential will be `nice', whose meaning will be clarified.

The class $\omega_e$ gives rise to a map 
\begin{align*}
   \iota_e: \R_{+} &\rightarrow T_{\Pic(Y)}\\
    t & \mapsto (t^{L_1 \cdot \omega_e}, \,t^{L_2 \cdot \omega_e},\, \cdots t^{L_r \cdot \omega_e})
\end{align*}
where $L_1, L_2, \cdots L_r$ is a fixed basis of $A_1(Y)$. 

Similarly, $\omega$ gives rise to a map $\iota_{\omega}: \R^{+} \rightarrow T_{\Pic(Y)}$ via $t \mapsto (t^{L_1 \cdot \omega}, \, \cdots t^{L_r \cdot \omega})$.

Notice that $\omega_b$ gives rise a map
\begin{align*}
    g_b: \R^{+} & \rightarrow T^D \\
    & t \mapsto (t^{\lambda_1}, \cdots, t^{\lambda_n}).
\end{align*}

Denote by $\vartheta_{i,t}$ the restriction of $\vartheta_i$ to the fiber $\cX_{\iota_e(t)}$. Let
\begin{equation}\label{eq:superpo}
    W_t = \sum_i t^{\lambda_i} \vartheta_{i,t}.
\end{equation}

The lemma below follows immediately from the definition of the action of $T^{D}$ and Theorem \ref{thm:rel-equi}.
\begin{lemma}
We have $g_b(t) \cdot \iota_e(t) = \iota_{\omega}(t)$
and $W_t$ is the superpotential on the fiber $\cX_{\iota_{\omega}(t)}$.
\end{lemma}

The fibers $\cX_{\iota_e(t)}$ and $\cX_{\iota_{\omega}(t)}$ are isomorphic, since $T^{D}$ acts on the compactified mirror family by changing the marking of the boundaries only. Thus, $W_t$ can be viewed either as regular functions on $\cX_{\iota_e (t)}$ or $\cX_{\iota_{\omega}(t)}$. Our convention is that $\cX_t=\cX_{\iota_\omega(t)}$ and $W_t\vert_{\cX_t}= \sum_i t^{\lambda_i} \vartheta_{i,t},$ where $t^{\lambda_i} \vartheta_{i,t}$ is the theta function on $\cX_t$. 

\begin{example}
Consider the Looijenga pair of $\mathbb{P}^2$ together with its toric boundary.  Then, condition \ref{eq:ample-con-1} is sufficient for $\omega_b$ to be ample. The image of $\iota_e$ is the identity element of $T_{\Pic(\mathbb{P}^2)} \simeq \C^{*}$. The superpotential restricted to $1 \in T_{\Pic(\mathbb{P}^2})$ is equal to the standard one $x+y+x^{-1}y^{-1}$. The action of $T^{D} \simeq (\C^{*})^3 $ rescales the coefficients before the monomials $x$, $y$, and $x^{-1} y^{-1}$ and we have 
\[
W_t = t^{\lambda_1}x + t^{\lambda_2}y + t^{\lambda_3} x^{-1}y^{-1}.
\]
\end{example}

\subsection{The expansion of the superpotential in the special chamber}
Recall that we denote by $\sigma_0$ the special chamber in $\bar{\D}'$ containing the origin. Below is the first key lemma in this subsection. 

\begin{figure}
    \centering
    \includesvg[width = 0.3\textwidth]{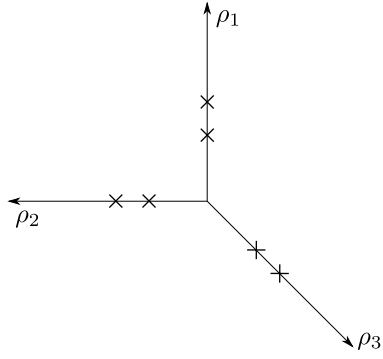}
    \caption{The toric model of $dP_3$}
    \label{fig:example_dP3_toric_model}
\end{figure}
\begin{lemma} \label{lem:expan-sup}
Suppose $(Y,D)$ is a Looijenga pair such that $Y$ is a del Pezzo surface of degree $\geq 3$ and $(Y,D)$ has a toric model. Fix a generic point $S$ in $\sigma_0$. Then, for each $i$, a broken line with the initial direction $n_i$ and the endpoint $S$ can only bend at the initial walls $\bar{\D}'_0$ of $\bar{\D}'$. Moreover, we can find a set of chambers in $\bar{\D}'$ including $\sigma_0$ such that these chambers are in 1-1 correspondence with toric charts in the toric blowup description of a smooth fiber in the mirror family. 
\end{lemma} 

\begin{proof}
We prove the cubic surface case with a toric model of $\mathbb{P}^2$ together with its toric boundary. Adapting a subset of arguments we use to prove this particular case will enable us the prove  the rest  cases since this case is the most complicated one. So, proving this case will be enough. 

See Figure \ref{fig:example_dP3_scatter} for part of the scattering diagram $\bar{\D}'$. 
\begin{figure}
    \centering
    \includesvg[width = 0.3\textwidth]{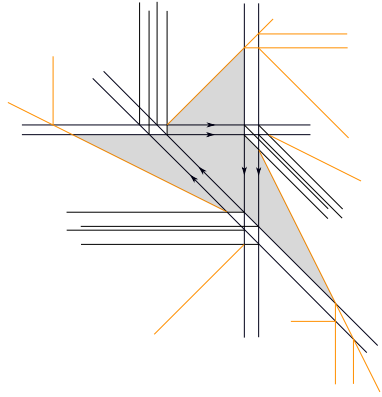}
    \caption{Part of the scattering diagram of $dP_3$. The initial walls and their first-order scatterings are in black.}
    \label{fig:example_dP3_scatter}
\end{figure}
Let us first prove the existence of chambers in the scattering diagram. We first look at the part of the scattering diagram that originated from the intersection of initial walls parallel to $\rho_{i-1}$ and $\rho_i\,(1\leq i \leq 3)$. It suffices to consider the case where $i=2$, as shown in Figure \ref{fig: scattering of two set of walls}. As shown in Example 1.6 of \cite{GPS}, all outgoing walls in this part of the scattering diagram are of one of the following forms:
\begin{enumerate}
    \item $\R_{\geq 0}\cdot(n, -n-1)$, \, $1+z^c x^n y^{-n-1},\,n \geq 1$.
    \item $\R_{\geq 0}\cdot (1,-1)$,\, $1-z^d xy$
    \item $\R_{\geq 0}\cdot (n+1, -n)$,\, $1+z^e x^{n+1}y^{-n},\, n \geq 1$.
\end{enumerate}
\begin{figure}
    \centering
    \includesvg[width = 0.22\textwidth]{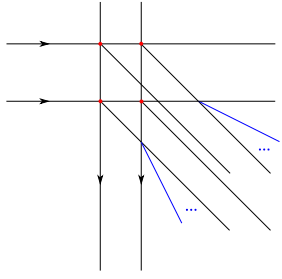}
    \caption{The scattering of initial walls parallel to $\rho_1$ and $\rho_2$}
    \label{fig: scattering of two set of walls}
\end{figure}
Here, $c,d$, and $e$ are elements in the monoid $\oplus_{i=1}^{2} (\mathbb{N}d_{i1}\oplus \mathbb{N}d_{i2})$ where $d_{ij}=\varphi(n_i)-[E_{ij}]$. As a result, only outgoing rays with slope $(n, -n-1)$ will intersect the two initial walls parallel to ray $\rho_3$. After these outgoing rays intersect the initial walls parallel to $\rho_3$, the outgoing walls emanating from these intersections will all have a slope between $(1,-1)$ and $(-1,0)$ (rotating clockwise). Then, it is clear to see that in $\bar{\D}'$, we have the desired set of chambers as shaded in gray in \ref{fig:example_dP3_scatter}. 

Next, let us prove the remaining part of the lemma regarding broken lines contributing to the expansion of the superpotential at $S$. Due to the symmetry of the scattering diagram, it suffices to prove that the statement holds for the expansion of $\vartheta_2$ at $S$. 

Let $\gamma$ be a broken line whose initial direction is $n_2$ and ending at $S$. First, if $\gamma$ bends, it is straightforward to see that the first bending of $\gamma$ has to occur at one of the initial walls parallel to $\rho_3$. Moreover, $\gamma$ has to pass through both initial walls parallel to $\rho_3$ and cannot bend at any other walls in between. After crossing initial walls parallel to $\rho_3$, if $\gamma$ has not entered $\sigma_0$, then $\gamma$ has to enter the region of the scattering diagram as shaded in green in Figure \ref{fig:example_dP3_scatter_2}. If $\gamma$ bends only at one of the initial walls parallel to $\rho_3$, the direction of $\gamma$ immediately after this bending will be $-n_1$ and $\gamma$ cannot enter $\sigma_0$ later. So, $\gamma$ has to bend at both of the initial walls parallel to $\rho_3$ before entering the region shaded in green with direction $-n_1 + 2n_2$. Then, it is straightforward to see that $\gamma$ cannot bend at any outgoing ray in the green region, and neither can it pass through the green region and enter $\sigma_0$ later. So, after entering the green region, $\gamma$ can only bend at the initial walls parallel to $\rho_1$ and then enter $\sigma_0$. So, the first statement lemma is also proved for this case. 
\end{proof}

\begin{figure}
    \centering
    \includesvg[width = 0.3\textwidth]{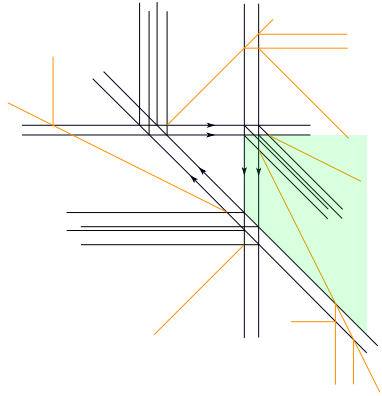}
    \caption{The green region that will trap $\gamma$}
    \label{fig:example_dP3_scatter_2}
\end{figure}

\begin{corollary}
The expansion of $W_t$ at a generic point in the special chamber $\sigma_0$ is a Laurent polynomial such that each of the monomials has a coefficient that is $t$ to a positive real power. 
\end{corollary}
\begin{proof}
The corollary follows immediately from Lemma \ref{lem:expan-sup} and conditions \ref{eq:ample-con-1} and \ref{eq:ample-con-3}.
\end{proof}

\begin{remark}
One can check that for each del Pezzo surface $(Y,D)$ of degree $\geq 3$,  $D$ can be chosen so that the expansion of the superpotential at $\sigma_0$ restricted to the fiber over $1 \in T_{\Pic(Y)}$  agrees with the standard one listed in, say, Table 1 of \cite{P-T_19}.
\end{remark}

\subsection{Tropicalization and amoebas}
\begin{definition}
The \emph{field of Hahn series} $K= \C \llbracket t^{\R} \rrbracket$ in the indeterminate $t$ over the field $\C$ and with value group $\R$ is the set of formal expressions of the form 
\[
c = \sum_{e \in \R} \mu_e t^e
\]
with $\mu_e \in \C$ such that the support 
\[
\mathrm{supp} := \{e \in \R : \mu_e \neq 0 \} 
\]
of $c$ is well-ordered. 

The \emph{valuation} $\nu(c)$ of a non-zero Hahn series $c = \sum_{e \in \R} \mu_e t^e \in K$ is defined as the smallest $e$ such that $\mu_e \neq 0$. This valuation makes $K$ a valued field with value group $\R$. 
\end{definition} 

\begin{definition}
Given a Laurent polynomial $f = \sum_{q\in N} c_q t^q$ in $K[N]$, we define the \emph{tropicalization}  $f^{\trop} : M_{\R} \rightarrow \R$ of $f$ to be the piecewise affine function 
\[
\min \{\mathrm{val}(c_q) + \langle q, \cdot \rangle \}.
\]
We define the \emph{tropical amoeba} to be the real hypersurface in $M_{\R}$ given by the non-smooth locus of $f^{\trop}$. 
\end{definition}

In particular, we can view the expansion of $W_t$ at a generic point in the special chamber $\sigma_{0}$ as an element in $K[N]$.

\subsection{Tropicalization of the superpotential}
In this subsection, we study the shape of the convex polygon 
\[
  \Xi := \{ m \in M_{\R} \mid   W_t^{\trop}(m) \geq 0  \}.
\]
In particular, $\Xi$ is the intersection of polygons $\bigcap_{i=1}^{n}  \Xi_i$ where 
\[
\Xi_i := \{ m \in M_{\R} \mid   \vartheta_{i,t}^{\trop}(m) \geq - \lambda_i \}
\]

For each $i$, without loss of generality, we could assume that the ordering of the distance between walls $\fd_{ij}$ and the origin is the opposite of the ordering of $c_{ij}$'s, i.e., the bigger the $c_{ij}$ is, the closer the wall $\fd_{ij}$ is to the origin.

\begin{lemma}\label{lem:max-bro}
Fix a generic point $S_0$ in $\sigma_0$. For each $i$, there exists a unique, maximally bent broken line $\gamma_i$ with initial direction $n_i$ and  ending at $S_0$ such that the faces of $\Xi_i$ are in 1-1 correspondence with monomials attached to each domain of linearity of $\gamma_i$. 
\end{lemma}
\begin{proof}
By Lemma \ref{lem:expan-sup}, any broken line contributing to $W_{t,S_0}$ has can only bend at an initial wall in $\bar{\D}'_0$. Also, observe that for any such broken line, when it bends at an initial wall, it can only hit the initial wall from the left side, i.e., the side that does not contain $\sigma_0$. As a result, for any broken line contributing to $W_{t,S_0}$,  $S_0$ is always on its left side and it always bends towards $S_0$. The existence of $\gamma_i$ thus follows. 

The second statement of the lemma follows from our assumption about the ordering of the distance between walls $\fd_{ij}$ and the origin. 
\end{proof}

    For each $i$, let $\vartheta^{*}_{i,t}$ be a truncation of $\vartheta_{i,t}$ defined as follows.  If there is at least one non-toric blowup on $\bar{D}_{i+1}$, let $\vartheta^{*}_{i,t}$ be the sum of all monomials attached to the truncated $\gamma_i$ after we truncate all the bending that happens after $\gamma_i$ crosses initial walls parallel to $\rho_{i+1}$. Otherwise,  just let $\vartheta^{*}_{i,t}$ be equal to the monomial $z^{n_i}$. 
Let $W_{t}^{*} = \sum_i t^{\lambda_{i}}\vartheta^{*}_{i,t}$ and 
\begin{equation}\label{eq:truncated_tropicalization_polytope}
     \Xi^* := \{ m \in M_{\R} \mid   (W_{t}^{*})^{\trop}(m) \geq 0  \}.
\end{equation}

\begin{lemma}
\label{lemma:Xi=Xi*}
  $\Xi^{*}$ is a bounded non-singular convex polytope and monomials in $W_{t}^{*}$ are in $1-1$ correspondence with faces of $\Xi$. Moreover, we have $\Xi = \Xi^{*}$.
\end{lemma}
\begin{proof}
That $\Xi^{*}$ is convex follows directly from the definition. Condition \ref{eq:ample-con-2} implies $\sum_i \lambda_i \bar{D}_i$ is an ample divisor on $\bar{Y}$. Then, if we only tropicalize the contribution of straight broken lines to $W_t$, we get a  bounded polytope, i.e., the moment polytope corresponding to $\sum_i \lambda_i \bar{D}_i$. Thus, that $\Xi^{*}$ is bounded  follows from the fact that it is a `chopping' of a bounded polytope. Condition \ref{eq:ample-con-5} implies that faces of $\Xi^{*}$ are in $1-1$ correspondence with monomials in $W_{t}^{*}$. That $\Xi^{*}$ is non-singular follows from that for each $i$, $c_{ij}\neq c_{ij'}$ if $j \neq j'$. The final statement that $\Xi = \Xi^{*}$ follows from Lemma \ref{lem:max-bro} and condition \ref{eq:ample-con-5}.

\end{proof}

\section{Oscillatory Integral on the Mirror Cycle of Structure Sheaf}

\label{sec:structuresheaf}

This section is devoted to the evaluation of the oscillatory integral on the cycle mirrored to the structure sheaf. We use the torus chart corresponding to the special chamber to do the computation. We choose a torus fibration of the torus chart, decompose the base space into several pieces, and evaluate the oscillatory integral on each piece.

For the rest of this paper, we fix a choice of $\omega = \omega_b+\omega_e = \sum_{i}\lambda_i D_i + \sum_{i,j}c_{ij} E_{ij}$ as in Section \ref{sec:super_superpotential}. We also denote the total transform $D_i+\sum_{j=1}^{l_i}E_{ij}$ by $D_i'$.

\subsection{Decomposition of the oscillatory integral}
\label{sec:structuresheaf_decomposition}
We denote the torus chart $\Spec\C[\Lambda_{\sigma_0}]$ in $\cX_{t}$ corresponding to the special chamber $\sigma_0$ by $U_{\sigma_0}$ and let $z_1=z^{(1,0)},z_2=z^{(0,1)}$ be its coordinates. Then the mirror cycle of the structure sheaf is the positive real locus of $U_{\sigma_0}$ in $\cX_{t}$, which we denote by $\Gamma_{\cO}$. The truncated theta functions $\vartheta^{*}_{i,t}$ restricted to $U_{\sigma_0}$ are
\begin{equation}\label{eq:truncated_theta_function_on_special_chamber}
    \vartheta^{*}_{i,t}|_{U_{\sigma_0}} = z^{n_i}\left(1+\sum\limits_{j=1}^{l_{i+1}}\prod\limits_{k=1}^{j}\left(t^{c_{(i+1)k}}z^{n_{i+1}}\right)\right),
\end{equation}
and thus we have that
\begin{equation}\label{eq:truncated_superpotential_on_special_chamber}
    (W_t^{*})^{\trop}|_{U_{\sigma_0}} = \min_{i\in\{1,\dots,n\}}\{\lambda_i+(\vartheta_{i,t}^*)^{\trop}\} = \min_{i\in\{1,\dots,n\}}\min_{j\in\{1,\dots,l_{i}\}}\{\beta_i(x),\beta_{ij}(x)\},
\end{equation}
where 
\begin{align*}
    \beta_i(x) &= \lambda_i+\langle x,n_i\rangle, \\
    \beta_{ij}(x) &= \lambda_{i-1}+\sum_{k=1}^jc_{ik}+\langle x,n_{ij}\rangle \text{ with } n_{ij}=n_{i-1}+jn_{i}.
\end{align*} 
We consider the torus fibration $\Log_t:U_{\sigma_0}\rightarrow M_{\R}\cong\R^2$ given by $\Log_t(z_1,z_2) = (\log_t(|z_1|),\log_t(|z_2|))$, then $ M_{\R}\times (M_{\R}/M)\cong U_{\sigma_0}$ through $(x,[x'])\mapsto \exp(\log t\cdot x+2\pi\sqrt{-1}x'))$. Since there is only one positive real point on each fiber $\Log_t^{-1}(x)$, we can regard $W_t|_{\Gamma_{\cO}}$ as a function on $M_{\R}$ and write it as 
\begin{equation}\label{eq:W_t_written_in_monomials}
    W_t(x) = \sum_{v\in V}t^{\beta_v(x)}+\sum_{\alpha}t^{\alpha(x)},\ \text{for $x\in M_\bR$},
\end{equation}
where $V=\{1,\dots,n\}\cup\bigcup\limits_{i=1}^n\{i1,i2,\dots,il_i\}$ and $\sum_{\alpha}t^{\alpha(x)} = (W_t-W_t^*)(x)$.

Now let us compute the oscillatory integral on $\Gamma_{\cO}$. Note that $\Omega|_{\Gamma_{\cO}}=(\log t)^2 dx_1\wedge dx_2$, so 
\begin{align*}
    Z_B(\Gamma_{\cO})=\int_{\Gamma_{\cO}}e^{-W_t}\Omega = (\log t)^2\int_{M_{\R}}e^{-W_t(x)}dx_1dx_2.
\end{align*}
We decompose $M_{\R}$ into several pieces and evaluate the integral on each piece. Set
\begin{align}\label{eq:polytope_P}
    \Polytope{\delta} :=& \{x\in M_{\R} \mid W_t^{\trop}|_{U_{\sigma_0}}(x)\geq\delta\}\nonumber\\
    =& \{x\in M_{\R} \mid (W_t^{*})^{\trop}|_{U_{\sigma_0}}(x)\geq\delta\}\nonumber\\
    =& \{x\in M_{\R} \mid \beta_v(x)\geq\delta \text{ for } v\in V\},\\
    \cPolytope{\delta} :=& \Log_t(\{z\mid W_t(z)\leq t^\delta\}\cap\Gamma_{\cO}).
\end{align}
Then $\Xi = \Xi^* = \Polytope{0}$ and the limit of $\cPolytope{\delta}$ is $\Polytope{\delta}$ when $t\rightarrow 0^+$.

\begin{lemma}[\cite{Mikhalkin_2004}, Corollary 6.4]
\label{lemma:mikhalkin}
    The region $\cPolytope{\delta}$ converges in the Hausdorff metric to $\Polytope{\delta}$ when $t\rightarrow 0^+$.
\end{lemma}

By choosing a small $\epsilon$ and when $t$ is small enough, we intend to decompose $Z_B(\Gamma_{\cO})$ as 
\begin{equation}
    \begin{split}
    Z_B(\Gamma_{\cO}) = (\log t)^2\left(\int_{\mathcal{P}_\epsilon}e^{-W_t(x)}dx_1dx_2 + \int_{M_{\R}\backslash\mathcal{P}_\epsilon}e^{-W_t(x)}dx_1dx_2\right).
    \end{split}
\end{equation}
The combinatorial information of $\Polytope{\delta}$ allows us to compute each part. On $\cPolytope{\epsilon}$, we measure the difference between $\vol(\cPolytope{\epsilon})$ and $\vol(\Polytope{\epsilon})$. On $M_{\R}\backslash\cPolytope{\epsilon}$, we foliate it with $\partial\cPolytope{\delta},\delta \leq \epsilon$ and measure the difference between $\vol_{\affine}(\partial\cPolytope{\delta})$ and $\vol_{\affine}(\partial\Polytope{\delta})$ using the result from \cite{A-G-I-S_2020}.

The proper choice of $\epsilon$ is as follows. For each $q\in V$ and each pair $\{q,q'\}\subset V$ with $n_q,n_{q'}$ adjacent, let us set
\begin{align}    
    \cB^{q} =& \left\{x\in \Polytope{\epsilon} \backslash \cPolytope{\epsilon} \middle\vert\ 
    \begin{aligned}
        \beta_q(x)\leq\beta_v(x)-\epsilon' \text{ for }v\in V\backslash \{q\}.\\
    \end{aligned}\right\}\\
    \cB^{q,q'} =& \left\{x\in \Polytope{\epsilon} \backslash \cPolytope{\epsilon} \middle\vert\ 
    \begin{aligned}
        \beta_{q'}(x)-\epsilon'<\beta_q(x)\leq\beta_{q'}(x),\\
        \beta_q(x)\leq\beta_v(x)-\epsilon' \text{ for }v\in V\backslash \{q,q'\}.\\
    \end{aligned}\right\}     
\end{align}
We choose $\epsilon'>\epsilon>0$ and a positive real number $b$ such that whenever $0<t<b$, we have
\begin{itemize}
    \item $\cB^q$'s and $\cB^{q,q'}$'s are disjoint from each other.
    \item $\Polytope{\epsilon} \backslash \cPolytope{\epsilon}$ is covered by $\cB^q$'s and $\cB^{q,q'}$'s.
    \item For any monomial $t^{\alpha(x)}$ in (\ref{eq:W_t_written_in_monomials}), there exists some $v\in V$ such that $t^{\alpha(x)}<t^{\epsilon'} t^{\beta_v(x)}$.
\end{itemize}
The existence of $\epsilon'$ for the last condition is because Lemma \ref{lemma:Xi=Xi*} and the generic choice of $\omega$. The existence of $b$ is guaranteed by Lemma \ref{lemma:mikhalkin}.

We will later compute the difference between $\vol(\Polytope{\epsilon})$ and $\vol(\cPolytope{\epsilon})$ by evaluating $\vol(\cB^{q})$'s and $\vol(\cB^{q,q'})$'s. Figure \ref{fig:Decomposition_of_base} illustrates a neighborhood of a corner of $\Polytope{\epsilon}$ and how we decompose $\Polytope{\epsilon}\backslash\cPolytope{\epsilon}$.

\begin{figure}
    \centering
    \includesvg[width = 0.8\textwidth]{Figures/Decomposition_of_base.svg}
    \caption{Decomposition of $\Polytope{\epsilon}\backslash\cPolytope{\epsilon}$}
    \label{fig:Decomposition_of_base}
\end{figure}

\subsection{Oscillatory integral on $\cPolytope{\epsilon}$}
We have $e^{-W_t(x)}=O(t^\epsilon)$ since $W_t(x)<t^\epsilon$ for $x\in\cPolytope{\epsilon}$, so 
\[
    \int_{\cPolytope{\epsilon}}e^{-W_t(x)}dx_1dx_2 = \left(1+O(t^\epsilon)\right)\vol(\cPolytope{\epsilon}).
\]
We want to evaluate it using the volume of $\Polytope{\epsilon}$, so we need to know $\vol(\Polytope{\epsilon})-\vol(\cPolytope{\epsilon})$.

\begin{proposition}\label{prop:difference_of_P_and_cP}
    The difference between the volumes of $\Polytope{\epsilon}$ and $\cPolytope{\epsilon}$ is
    \begin{align*}
        \vol(\Polytope{\epsilon})-\vol(\cPolytope{\epsilon}) = \frac{1}{(\log t)^2}\left((n+\sum_{i=1}^{n}l_i)\zeta(2)+O(t^\epsilon)\right).
    \end{align*}
\end{proposition}

\begin{proof}
    We compute the affine volumes of $\cB^{q}$'s and $\cB^{q,q'}$'s respectively. For each $\cB^{q}$, choose a primitive integer vector $\gamma_q$ such that $n_q\wedge\gamma_q=1$ and let $c_q(x)=\langle x,\gamma_q\rangle$. Let $C_1=\min_{x\in \cB^{q}}\{c_q(x)\}$ and $C_2=\max_{x\in \cB^{q}}\{c_q(x)\}$. For each $c\in[C_1,C_2]$, let $x_1^{q}(c),x_2^{q}(c)$ be the two endpoints of $\{x \mid c_q(x)=c\}\cap \cB^{q}$, then
    \begin{align*}
        \vol(\cB^{q}) =& \int_{\cB^{q}} d\beta_q\wedge dc_q\\
        =& \int_{C_1}^{C_2}\left|\beta_q(x_1^{q}(c))-\beta_q(x_2^{q}(c))\right|dc.
    \end{align*}
    We know $\beta_q(x)=\epsilon + O(t^{\epsilon'})$ for $x\in \cB^{q}$ because $W_t|_{\cB^{q}}(x)=t^{\beta_q(x)}\left(1+O(t^{\epsilon'})\right)>t^\epsilon$ and $\beta_q(x)>\epsilon$. Then $|\beta_q\left(x_1^{q}(c)\right)-\beta_q\left(x_2^{q}(c)\right)|=O(t^{\epsilon'})$, and 
    \begin{align*}
        \vol(\cB^{q}) = O(t^{\epsilon'}) = \frac{1}{(\log t)^2}O(t^{\epsilon}).
    \end{align*}
    The region $\cB^{q,q'}$is bounded by $(\beta_{q'}-\beta_q)(x)=0$ and $(\beta_{q'}-\beta_q)(x)=\epsilon'$. For $c\in[0,\epsilon']$, let $x_1^{q,q'}(c), x_2^{q,q'}(c)$ be the two endpoints of $\{x \mid (\beta_{q'}-\beta_q)(x)=c\}\cap \cB^{q,q'}$. Then
    \begin{align*}
        \vol(\cB^{q,q'}) =& \int_{\cB^{q,q'}} d\beta_q\wedge d(\beta_{q'}-\beta_q)\\
        =&\int_{0}^{\epsilon'} \left|\beta_q(x_2^{q,q'}(c))-\beta_q(x_1^{q,q'}(c))\right| dc.
    \end{align*}
     Since $\beta_q(x_1^{q,q'}(c))=\epsilon$ and $W_t(x_2^{q,q'}(c)) = t^{\beta_q(x_2^{q,q'}(c))}\left(1+t^c+O(t^{\epsilon'})\right)=t^\epsilon$, we have $\left|\beta_q(x_2^{q,q'}(c))-\beta_q(x_1^{q,q'}(c))\right| = -\log_t\left(1+t^c+O(t^{\epsilon'})\right)$, and thus 
    \begin{align*}
        \vol(\cB^{q,q'})=&\int_{0}^{\epsilon'}-\log_t(1+t^c+O(t^{\epsilon'}))dc\\
        =&\int_{0}^{\epsilon'}-\log_t(1+t^c)dc+\int_{0}^{\epsilon'}\log_t(1+t^c)-\log_t(1+t^c+O(t^{\epsilon'}))dc\\
        =&\frac{1}{(\log t)^2}\int_{1}^{t^{\epsilon'}}-\frac{\log(1+u)}{u}du+\frac{O(t^{\epsilon'})}{\log t}\\
        =&\frac{1}{(\log t)^2}\left(\sum_{l=1}^\infty \frac{(-1)^{l+1}}{l^2}(1-t^{l\epsilon'})+\log t O(t^{\epsilon'})\right)\\
        =&\frac{1}{(\log t)^2}\left(\frac{1}{2}\zeta(2)+O(t^{\epsilon})\right).
    \end{align*}
    Now since the number of $\cB^{q}$'s is $n+\sum_{i=1}^nl_i$ and the number of $\cB^{q,q'}$'s is $2(n+\sum_{i=1}^nl_i)$, we have   \begin{align*}
        \vol(\Polytope{\epsilon})-\vol(\cPolytope{\epsilon})
        = \frac{1}{(\log t)^2}\Bigl((n+\sum_{i=1}^{n}l_i)\zeta(2)+O(t^\epsilon)\Bigr).
    \end{align*}
\end{proof}
Using Proposition \ref{prop:difference_of_P_and_cP},  
\begin{equation}\label{eq:Z_topo(O)_on_P(epsilon)}
    \int_{\cPolytope{\epsilon}}e^{-W_t}dx_1dx_2   =\vol(\Polytope{\epsilon})-\frac{1}{(\log t)^2}\Bigl((n+\sum_{i=1}^{n}l_i)\zeta(2)+O(t^\epsilon)
    \Bigr).
\end{equation}

\subsection{Oscillatory integral on \texorpdfstring{$M_{\R}\backslash\mathcal{P}_\epsilon$}{}} 
Note that $M_{\R}\backslash\mathcal{P}_\epsilon = \bigcup_{\delta\in[\epsilon,-\infty)}\partial\cPolytope{\delta}$ and $W_t(x)|_{\partial\cPolytope{\delta}} = t^\delta$, so 
\[
    \int_{M_{\R}\backslash\mathcal{P}_\epsilon}e^{-W_t(x)}dx_1dx_2 = \int_{\epsilon}^{\infty} -e^{-t^\delta}\vol_{\affine}(\partial\cPolytope{\delta})d\delta
\] 
with $\vol_{\affine}(\partial\cPolytope{\delta}) = -\frac{\partial\vol\cPolytope{\delta}}{\partial\delta}$. We want to compare $\vol_{\affine}(\partial\cPolytope{\delta})$ with $\vol_{\affine}(\partial \Polytope{\delta})$ using the result from \cite{A-G-I-S_2020}.

\begin{proposition}[\cite{A-G-I-S_2020}, Section 2.1]\label{prop:affine_volume_of_boundary_P(delta)}
   Let $\vol_{\affine}(\partial \cPolytope{\delta}) = -\frac{\partial\vol\cPolytope{\delta}}{\partial\delta}$ and $\vol_{\affine}(\partial \Polytope{\delta}) = -\frac{\partial\vol \Polytope{\delta}}{\partial\delta}$ be the affine volumes of $\partial \cPolytope{\delta}$ and $\partial \Polytope{\delta}$, then for $\delta \leq \epsilon$,
   \[
        \vol_{\affine}(\partial \cPolytope{\delta}) = \vol_{\affine}(\partial \Polytope{\delta}) + O(t^\epsilon).
   \]
\end{proposition}

Using Proposition \ref{prop:affine_volume_of_boundary_P(delta)}, we have that
\begin{equation}\label{eq:Z_topo(O)_on_P(epsilon)_complement}
    \begin{split}
        \int_{M_{\R}\backslash\cPolytope{\epsilon}}e^{-W_t}dx_1dx_2 = (1+O(t^\epsilon))\int_{\epsilon}^{\infty} -e^{-t^\delta}\vol_{\affine}(\partial \Polytope{\delta})d\delta.
    \end{split}
\end{equation}

\subsection{Conclusion}
We need to compare $Z_B(\Gamma_{\cO})$ with $Z_{\topo}(\cO)$ as in Conjecture \ref{conj:main}. Using \begin{equation}\label{eq:Gamma_class}
    \hat \Gamma_{Y} =\exp\left(-\gamma c_1(Y)+\sum_{k=2}^{\infty}\zeta(k)(k-1)!\ch_k(TY)\right),
\end{equation} we have that
\begin{equation}\label{eq:A(O)}
    \begin{split}
    Z_{\topo}(\cO) =& \frac{1}{2}\omega^2(\log t)^2 + \gamma\omega\cdot c_1(Y)\log t + \zeta(2)\ch_2(TY)+\frac{1}{2}\gamma^2c_1(Y)^2\\
    =& \frac{1}{2}\omega^2(\log t)^2 + \gamma\omega\cdot c_1(Y)\log t + \zeta(2)\sum\limits_{i=1}^n\left(\frac{1}{2}c_1(Y)^2-c_2(Y)\right) + \frac{1}{2}\gamma^2c_1(Y)^2.
    \end{split}
\end{equation}
This can be interpreted using the combinatorial information of $\Polytope{\delta}$.

\begin{lemma}\label{lemma:combinatorial_information_of_polytope}
For $\delta \leq \epsilon$, the volume of $\Polytope{\delta}$ is
\begin{align*}
    V(\delta):=\vol(\Polytope{\delta}) = V(0)+\delta \frac{\partial V}{\partial\delta}(0) +\frac{1}{2}\delta^2 \frac{\partial^2 V}{\partial\delta^2}(0),\\
    \text{with }V(0) = \frac{1}{2}\omega^2, \frac{\partial V}{\partial\delta}(0) = -\omega\cdot c_1(Y), \frac{\partial^2 V}{\partial\delta^2}(0) = c_1(Y)^2.
\end{align*}
\end{lemma}

\begin{proof}
    It is sufficient to show that $V(\delta) = \frac{1}{2}(\omega-\delta c_1(Y))^2$. Let $\omega_\delta' = \sum_{i=1}^n(\lambda_i-\delta)D_i'$ and $\Polytope{\delta}' = \{x\mid\beta_i(x)\geq\delta,\forall i \in\{1,\dots,n\}\}$. Then 
    \begin{align*}
        \vol(\Polytope{\delta}') = \frac{1}{2}{\omega'_\delta}^2
    \end{align*}
    according to Duistermaat–Heckman theorem. Let $\triangle_{ij}$ be the triangle bounded by $\beta_{i}(x)=\delta,\beta_{i(j-1)}(x)=\delta,\beta_{ij}(x)=\delta$, where we set $\beta_{i 0}(x)$ to be $\beta_{i-1}(x)$. Then 
    \begin{align*}
        \Polytope{\delta}' = \Polytope{\delta}\cup\bigcup_{i=1}^n\bigcup_{j=1}^{l_i}\triangle_{ij}.
    \end{align*}
    So we have that 
    \begin{equation*}
        \vol(\Polytope{\delta}') = \vol(\Polytope{\delta})+\sum_{i=1}^n\sum_{j=1}^{l_i}\vol(\triangle_{ij}).
    \end{equation*}
    Note that 
    \begin{align*}
        \frac{1}{2}{\omega'_\delta}^2=\frac{1}{2}(\omega-\delta c_1(Y))^2+\frac{1}{2}\sum_{i=1}^n\sum_{j=1}^{l_i}(\lambda_i-c_{ij}-\delta)^2,
    \end{align*}
    so we have $\vol(\Polytope{\delta}') = \frac{1}{2}(\omega-\delta c_1(Y))^2$ since $\vol(\triangle_{ij}) = \frac{1}{2}(\lambda_i-c_{ij}-\delta)^2$.
\end{proof}

Using Lemma \ref{lemma:combinatorial_information_of_polytope}, we have that
\begin{equation}\label{eq:ZA(O)}
    Z_{\topo}(\cO) = V(0)(\log t)^2-\gamma \frac{\partial V}{\partial\delta}(0)\log t +\frac{1}{2}(\gamma^2+\zeta(2))\frac{\partial^2 V}{\partial\delta^2}(0) - (n+\sum\limits_{i=1}^nl_i)\zeta(2),
\end{equation}
where we use the fact that $c_2(Y) = c_2(\bar{Y})+\sum_{i=1}^nl_i$ and $\Bar{Y}$ is smooth.

For $Z_B(\Gamma_{\cO})$, using (\ref{eq:Z_topo(O)_on_P(epsilon)}) and (\ref{eq:Z_topo(O)_on_P(epsilon)_complement}), we have that
\begin{equation}\label{eq:Z_topo(Gamma_O)}
    \begin{split}
        &Z_B(\Gamma_{\cO})\\ 
        =& (\log t)^2\vol(\Polytope{\epsilon})-(n+\sum_{i=1}^{n}l_i)\zeta(2)+O(t^\epsilon)+(1+O(t^\epsilon))(\log t)^2\int_{\epsilon}^{\infty} -e^{-t^\delta}\vol_{\affine}(\partial \Polytope{\delta})d\delta\\
        =& V(\epsilon)(\log t)^2+(1+O(t^\epsilon))(\log t)^2\int_{\epsilon}^{-\infty} e^{-t^\delta}\frac{\partial V}{\partial\delta}(\delta)d\delta-(n+\sum_{i=1}^{n}l_i)\zeta(2)+O(t^\epsilon).
    \end{split}
\end{equation}
Since
\begin{align*}
    \log t\int_{\epsilon}^{-\infty}e^{-t^\delta}d\delta =& -\gamma-\epsilon\log t+O(t^\epsilon),\\
    (\log t)^2\int_{\epsilon}^{-\infty}e^{-t^\delta}\delta d\delta =& \frac{1}{2}\gamma^2+\frac{1}{2}\zeta(2)-\frac{1}{2}\epsilon^2(\log t)^2+O(t^\epsilon),
\end{align*}
we have $Z_B(\Gamma_{\cO}) = Z_{\topo}(\cO)+O(t^\epsilon)$, which proves Conjecture \ref{conj:main} for the case of structure sheaf.

\section{Mirror Cycle Construction}

\label{sec:mirror_cycle}

In this section we construct a $2$-cycle $\Gamma_{\cO(L)}$ in $\cX_{t}$ which is mirror to a line bundle $\cO(L)$ over the del Pezzo surface $(Y,D)$. Recall that $(Y,D)$ admits a toric model. We call $L$ \textit{toric} if it is the pullback of a linear combination of toric divisors, i.e., $L = \sum_{i=1}^na_iD_i'$ with $D_i'=p^{*}(\Bar{D}_i)$ and $a_i\in\Z$. We call $L$ \textit{exceptional} if it is a linear combination of exceptional divisors, i.e., $L = \sum_{i=1}^{n}\sum_{j=1}^{l_i}b_{ij}E_{ij}$ with $b_{ij}\in\Z$. The construction then depends on the type of $L$.

We need some preparations before the construction. Let $\pi:M_{\R} \times M_{\R} \rightarrow U_{\sigma_0}, (x,x') \mapsto \exp(\log t\cdot x+2\pi\sqrt{-1}x')$ be the universal cover of $U_{\sigma_0}$ and $\Polytope{\epsilon'}\subset M_{\R}$ be the polytope determined by $\beta_v(x) = \epsilon'$, $v\in V=\{1,\dots,n\}\cup\bigcup_{i=1}^n\{i1,i2,\dots,il_i\}$ as in (\ref{eq:polytope_P}). For each $n_i$, let $F_i$ be the face of $\Polytope{\epsilon'}$ which is determined by $\beta_i(x) = \epsilon'$ and $R_i=\{x\in M_{\R}\backslash \Polytope{\epsilon'} \mid \beta_i(x) < \beta_v(x)-\epsilon'\text{ for }v\in V\backslash\{i\}\}$ be an unbounded region in $M_{\R}\backslash\Polytope{\epsilon'}$. We denote $\nDual{i}\in M$ by the identification $\langle \nDual{i},- \rangle = n_i\wedge -$, where the affine volume on the right-hand side is given by the lattice $N\in N_\bR$.

\subsection{Toric divisor}
Suppose $L = \sum_{i=1}^na_iD_i'$, and let $\varphi_L:N_{\R}\rightarrow\R$ be the piecewise linear function induced by $\varphi_L(n_i)=a_i$. Let $\sigma_{i,i+1}$ be the cone between $\R_{\geq 0}\cdot n_i$ and $\R_{\geq 0}\cdot n_{i+1}$, $\varphi_{L,i,i+1}$ be the linear function extended from $\varphi_L|_{\sigma_{i,i+1}}$. Then for $n\in N_{\R}$, we have $\varphi_{L,i,i+1}(n)-\varphi_{L,i-1,i}(n) = \langle \kappa_i \nDual{i}, n\rangle$ for some constant $\kappa_i$. We call $\kappa_i$ the \textit{kink} of $\varphi_L$ along the ray $\R_{\geq 0}\cdot n_i$. Since $\Bar{Y}$ is smooth in the toric model $p:Y\rightarrow\Bar{Y}$, we have 
\begin{align*}
    \kappa_i =& \varphi_{L,i,i+1}(n_{i+1})+\varphi_{L,i-1,i}(-n_{i+1})\\
    =& a_{i+1}+\varphi_{L,i-1,i}(\Bar{D}_i^2n_i+n_{i-1})\\
    =& a_{i+1}+\Bar{D}_i^2a_i+a_{i-1}\\
    =& D_i'\cdot L,
\end{align*}
where we use $n_{i-1}+D_i^2n_i+n_{i+1} = 0$ for the second equation. So we have $\sum_{i=1}^{n}\kappa_i \nDual{i}=0$ and there is a twisted polytope $\sigma_L$ in $M_{\R}$ with integer vertices $v_{L,1}, v_{L,2}, \dots, v_{L,n}$ such that $v_{L,i+1}-v_{L,i}=\kappa_i\nDual{i}$ in the sense of \cite{K-T_1993}.

The mirror cycle will then be constructed using the information of $\varphi_L$  and a choice of a path $\Path_i$ in $M_{\R}$. Let $\Path_i:[0,\infty)\rightarrow M_{\R}$ be a smooth path such that
\begin{itemize}
    \item $\Path_i(0) = 0$, $\Path_i(1)\in F_i$ and $\Path_i(s) \in \Polytope{\epsilon'} $ for $s\in[0,1]$.
    \item $\Path_i(s) \in R_i$ for $s>1$ and $\lim\limits_{s\rightarrow\infty}\beta_i\left(\Path_i(s)\right)=-\infty$.
\end{itemize}
We also denote $\Path_i|_{[0,1]}$ and $\Path_i|_{(1,\infty)}$ by $\Path_{i,\In}$ and $\Path_{i,\Out}$ respectively. Then the mirror cycle is 

\begin{equation}
    \Gamma_{\cO(L)} = \Gamma_{\cO} + \sum\limits_{i=1}^{n}\Gamma_{i} + \Gamma_{\Center},
\end{equation}
where
\begin{align*}
    &\Gamma_{i} = \kappa_i \pi(\Path_i \times \nDual{i})\subset U_{\sigma_0}\subset \cX_{t},\\
    &\Gamma_{\Center} = \pi(\{0\}\times \sigma_L)\subset U_{\sigma_0}\subset \cX_{t},
\end{align*}
and we regard them as $2$-chains in $ \cX_{t}$. Here the twisted polytope $\sigma_L$ as a chain comes equipped with degrees on each of its supporting convex polytope. When $L$ is ample the twisted polytope is a convex polytope with degree $1$.

\begin{remark}
If we choose an isomorphism $\ast:N_{\R}\rightarrow M_{\R}$ which is induced by an inner product on $N_{\R}$ and replace $\Gamma_i$ with $\kappa_i \pi(\R_{\leq 0}\cdot n_i^\Dual \times \nDual{i})$, then it would be the characteristic cycle of a constructible sheaf mirror to $\cO(L)$ as described in \cite{F-L-T-Z_2014}. The reason why we do not explicitly use the characteristic cycle is that we are dealing with the oscillatory integrals so the actual symplectic structure in the Landau-Ginzburg model does not matter. 
\end{remark}

\subsection{Exceptional divisor}\label{sec:mirror_cycle_exceptional}
Firstly, for each $n_i$, let $\fd_{i1},\fd_{i2},\dots,\fd_{il_i}$ be the initial walls in the scattering diagram $\bar{\D}'$ which are parallel to $n_i$ and they are ordered such that $\fd_{i(j+1)}$ is on the right side of $\fd_{ij}$. Then for each $\fd_{ij}$, there is a local chart 
$U_{\fd_{ij}}$ of $\cX_{t}$ which is 
\begin{align*}
    U_{\fd_{ij}} = \{(u,v,w)\in\C^2\times\Spec\C[\Lambda_{\mathfrak{d}_{ij}}]\mid uv = H_{ij}(w)\},
\end{align*}
where $\Lambda_{\mathfrak{d}_{ij}} = N \cap \R\cdot n_i$ is a sublattice of $N$ and $H_{ij}(w)\in\C[\Lambda_{\mathfrak{d}_{ij}}]$ is given by $H_{ij}(w) = \prod_{q=1}^{j}f_{\fd_{iq}}|_{\cX_t} = \prod_{q=1}^{j}(1+t^{c_{iq}}w)$. 

We have $U_{\sigma_0}\hookrightarrow U_{\fd_{ij}}$ which is induced by $u\mapsto z^{\gamma_i}, w\mapsto z^{n_i}$, where $\gamma_i$ is a primitive integer vector in $N$ such that $n_i\wedge\gamma_i = 1$. Note that $U_{\fd_{ij}} \backslash U_{\sigma_0} = \bigcup_{q=1}^{j}\{(0,v,w) \mid v\in\C, w =-t^{-c_{iq}}\}$ is a disjoint union of $j$ copies of $\C$. The mirror cycle $\Gamma_{\cO(L)}$ will then be constructed in $U_{\fd_{ij}}$.

Suppose $L = bE_{ij}$ with $b\in\Z$. For each $k\in\{1,\dots,|b|\}$, let us choose a smooth path $p_{k}(s), s\in[0,1]$ in $\Spec\C[\Lambda_{\mathfrak{d}_{ij}}] \cong \C^*$ such that 
\begin{itemize}
    \item $p_{k}(0)=t^{-c_{ij}}$ and $p_{k}(1)=-t^{-c_{ij}}$.
    \item $p_{k}(s)\neq-t^{-c_{iq}}$ for $s\in[0,1)$ and $q\in\{1,\dots,j\}$.
    \item The argument of $p_{k}(s)$ is changed by $(1-2k)\pi$.
    \item $t^{-A(\epsilon-\epsilon')} < \|p_{k}(s)\|t^{c_{ij}} < t^{A(\epsilon-\epsilon')}$.
\end{itemize}
The last condition is for the purpose of the estimate of $W_t$ on the mirror cycle in Lemma \ref{lemma:Gamma_sing_ij_monomial_small}, where $A$ is a positive real number such that
\[
    0<A<\min_{\{l_1\in\Z_{>0},l_2\in\Z_{\neq 0}\}}\big\{\frac{l_1}{\left|l_2\right|} \mid \text{ There is a } z^{l_2n_i}\text{-monomial in }(W_t)^{l_1}.\big\}
\]
Such $A$ exists because there are finite monomials in $W_t$, thus there is an upper bound for $\frac{|\Proj_{n_i}n|}{|n_i|}$ for a $z^n$-monomial in $W_t$, where $\Proj_{n_i}$ is taken with respect to an arbitrary inner product on $N_{\R}$. Figure \ref{fig:path_in_C^*} shows an example of $p_k(s)$ when $k=2$.

\begin{figure}
    \centering
    \includesvg[width = 0.5\textwidth]{Figures/Path_in_C_star.svg}
    \caption{A smooth path in $\Spec\C[\Lambda_{\mathfrak{d}_{ij}}]$ for the construction of $\Gamma_{\cO(2E_{ij})}$}
    \label{fig:path_in_C^*}
\end{figure}

The mirror cycle will then be constructed using $p_k(s)$ and a choice of path $\Path_{ij}$ in $M_{\R}$ as follows. Let $\fb_{ij}$ be the line in $M_{\R}$ determined by $\langle x,n_i \rangle+c_{ij} = 0$, then $\fb_{ij}$ has affine distance $c_{ij}$ to the origin and it passes through the vertex $v_{ij}$ of $\Polytope{\epsilon'}$ which is determined by $\beta_{i-1}(x) = \beta_{ij}(x) = 0$. Then let us choose a smooth path $\Path_{ij}:[0,\infty)\rightarrow M_{\R}$ such that
\begin{itemize}
    \item $\Path_{ij}(0)\in\fb_{ij}$, $\Path_{ij}(1) \in F_i$ and $\Path_{ij}(s)\in \Polytope{\epsilon'}$ for $s\in[0,1]$.
    \item $\Path_{ij}(s)\in R_{i}$ for $s>1$ and $\lim\limits_{s\rightarrow\infty}\beta_i\left(\Path_i(s)\right)=-\infty$.
\end{itemize}
We also denote $\Path_{ij}(0),\Path_{ij}|_{[0,1]}$ and $\Path_{ij}|_{(1,\infty)}$ by $x_{ij}, \Path_{ij,\In}$ and $\Path_{ij,\Out}$ respectively. Then the mirror cycle is 

\begin{equation}\label{eq:mirror_cycle_except}
    \Gamma_{\cO(L)} = \Gamma_{\cO} + \Gamma_{ij} + \Gamma_{\sing,ij},
\end{equation}
where
\begin{align*} 
    \Gamma_{ij} &= b\pi(\Path_{ij}\times \nDual{i})\subset U_{\sigma_0}\subset \cX_{t},\\
    \Gamma_{\sing,ij} &= \sign(b)\sum\limits_{k=1}^{|b|}\Gamma_{\sing,ij,k} \text{ with }\\
    \Gamma_{\sing,ij,k} &= \big\{w = p_{k}(s),u=\|H_{ij}(p_{k}(s))\|^{s/2}\cdot \exp\left( (2\pi\sqrt{-1})\tau \right),\\
    &v=u^{-1}H_{ij}(p_{k}(s)) \mid  (s,\tau)\in[0,1]^2 \big\}\subset U_{\fd_{ij}}\subset \cX_{t},
\end{align*}
and we regard them as $2$-chains in $ \cX_{t}$. Note that for each $\Gamma_{\sing,ij,k}$, it is equal to $\pi(\{x_{ij}\}\times \nDual{i})$ when $s=0$ and it is a single point $(u,v,w) = (0,0,-t^{-c_{ij}})$ when $s=1$, which is the only point not contained in $U_{\sigma_0}$. The choice of the norm $\|u\| = \|H_{ij}(p_{k}(s))\|^{s/2}$ is to guarantee the continuity at $(u,v,w) = (0,0,-t^{-c_{ij}})$. The choice is not essential, i.e., any other choice that guarantees continuity could work as it does not affect the result of the oscillatory integral on the cycle.

\subsection{Conclusion}
For a general divisor $L=L_{\toric}+L_{\Except}$ with $L_{\toric}$ toric and $L_{\Except} = \sum_{i=1}^n\sum_{j=1}^{l_i} b_{ij}E_{ij}$ exceptional, the mirror cycle is
\begin{equation}\label{eq:mirror_cycle_of_a_general_divisor}
    \Gamma_{\cO(L)} = \Gamma_{\cO} + (\Gamma_{\cO(L_{\toric})} - \Gamma_{\cO}) + \sum_{i=1}^n\sum_{j=1}^{l_i} (\Gamma_{\cO(b_{ij}E_{ij})} -\Gamma_{\cO}). 
\end{equation} 

This assignment gives a morphism $\alpha:K(Y)\to H_2(\cX_{t}; W_t^{-1}(+\infty))$ by $\cO(L)\mapsto \Gamma_{\cO(L)}$. We expect it to be an isomorphism, and the product structure in the K-group should match the convolution production in the fiber direction.
 
\begin{remark}
Our construction of mirror cycles should be related to the construction in \cite{Hacking-Keating_2022}.    
\end{remark}

We elaborate on the images of different pieces of $\Gamma_{\cO(L)}$ under $\Log_t:U_{\sigma_0}\rightarrow M_{\R}$, as we need to use them to evaluate the oscillatory integrals in Section \ref{sec:anybundle}. We have that
\begin{align*}
    &\Log_t(\Gamma_{\cO}\cap U_{\sigma_0}) = M_{\R},\\
    &\Log_t(\Gamma_i) = \Path_i \text{ when }\Gamma_i\neq\emptyset,\\
    &\Log_t(\Gamma_{\Center}) = 0,\\
    &\Log_t(\Gamma_{ij}) = \Path_{ij},\\
    &\Log_t(\Gamma_{\sing,ij,k}\cap U_{\sigma_0}) =: \Path_{\sing,ij,k},
\end{align*}
The first four are clear from the constructions, and for $\Path_{\sing,ij,k}$, it is a curve depending on the choice of $p_k(s)$. It starts from $x_{ij}$, converges to $\fb_{ij}$ near infinity and lives in a small neighborhood of $\fb_{ij}$. Note that the only point in $\Gamma_{\sing,ij,k}$ which is not contained in $U_{\sigma_0}$ is $(u,v,w) = (0,0,-t^{-c_{ij}})$ in $U_{\fd_{ij}}$, and we can regard its $\Log_t$-image as the infinity point along $\fb_{ij}$.

\begin{example}\label{example_dP5_mirror_cycles}
\begin{figure}
    \centering
    \includesvg[width = 0.3\textwidth]{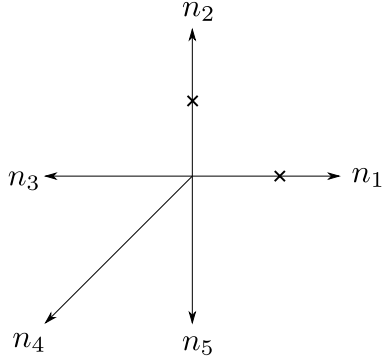}
    \caption{The toric model of $dP5$}
    \label{fig:example_dP5_toric_model}
\end{figure}
\begin{figure}
    \centering
    \includesvg[width = 1.0\textwidth]{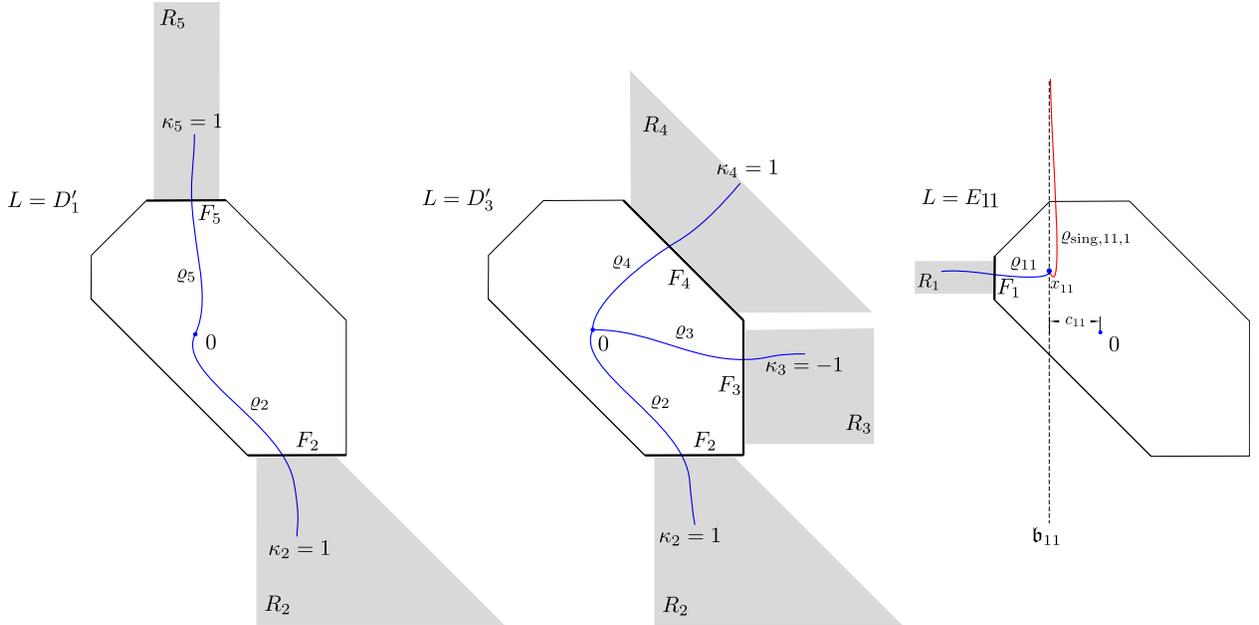}
    \caption{$\Log_t$-image of $\Gamma_{\cO(L)}$ for $L=D_1',D_3',E_{11}$ in a del Pezzo 5 surface}
    \label{fig:example_dP5_mirror_cycle}
\end{figure}
Let $(Y,D)$ be a del Pezzo $5$ surface which admits a toric model as shown in Figure \ref{fig:example_dP5_toric_model}. Let $D_i',i=1,2,\dots,5$ be the total transform of $\bar{D}_i$ and $E_{11}, E_{21}$ be the exceptional divisors blown up on $\bar{D}_1,\bar{D}_2$. Then the $\Log_t$-images of different pieces of the mirror cycles $\Gamma_{\cO(L)}$ for $L=D_1',D_3',E_{11}$ are indicated by Figure \ref{fig:example_dP5_mirror_cycle}, where $F_i$ denotes the face of $\Polytope{\epsilon'}$ determined by $\beta_{i} = \epsilon'$. Note that for the total transforms $D_1'$ and $D_3'$, $\Gamma_i$ is determined by $\kappa_i$, which is equal to its intersection number with $D_i'$. For the exceptional divisor $E_{11}$, the $\Log_t$-images of $\Gamma_{11}$ and $\Gamma_{\sing,11,1}$ are the blue and red curves. 

\end{example}

\section{Oscillatory Integrals on the Mirror Cycles}

\label{sec:anybundle}

In this section we compute the oscillatory integral $Z_B(\Gamma_{\cO(L)})$ on the mirror cycle $\Gamma_{\cO(L)}$. 
The computation can be restricted to $U_{\sigma_0}$. So we consider the torus fibration $\Log_t:U_{\sigma_0}\rightarrow M_{\R}$ and use the $\Log_t$-images of pieces of $ \Gamma_{\cO(L)}$ to evaluate the oscillatory integral on each piece. Since we have done it for $\Gamma_{\cO}$, it is sufficient to do it for $\Gamma_i,\Gamma_{\Center},\Gamma_{ij}$ and $\Gamma_{\sing,ij,k}$. We will use $\Omega|_{U_{\sigma_0}} = d\log z_1\wedge d\log z_2 = d\log(z^{n_i})\wedge d\log(z^{n'})$ for any $n'\in N$ such that $n_i\wedge n' = 1$. In particular, if we represent $\Omega$ using the coordinates $(u,v,w)$ of $U_{\fd_{ij}}$, we have $\Omega|_{U_{\sigma_0}} = d\log w\wedge d\log u$.

\subsection{Oscillatory integral on $\Gamma_i$, $\Gamma_{ij}$ and $\Gamma_{\mathrm{center}}$.}

Let us set 
\begin{align*}
    \Gamma_{v,\In} =& \Gamma_v\cap\Log_t^{-1}(\Path_{v,\In}),\Gamma_{v,\Out} = \Gamma_v\cap\Log_t^{-1}(\Path_{v,\Out})
\end{align*}
for $v\in V=\{1,\dots,n\}\cup\bigcup_{i=1}^n\{i1,i2,\dots,il_i\}$.
Then 
\begin{equation*}
    \int_{\Gamma_{v}}e^{-W_t}\Omega = \int_{\Gamma_{v,\In}}e^{-W_t}\Omega + \int_{\Gamma_{v,\Out}}e^{-W_t}\Omega.
\end{equation*}

We need to evaluate $W_t$ for $z\in\Gamma_{v,\In}$ and $z\in\Gamma_{v,\Out}$.

\begin{lemma}\label{lemma:Gamma_O_L_W_approximation}
For $v=i\text{ or }ij$, if we write $W_t|_{U_{\sigma_0}}$ as $W_t(z) = t^{\lambda_i}z^{n_i}(1+h(z))$, then 
    \begin{align*}
        \|W_t(z)\| =& O(t^{\epsilon'})\text{ for }z\in\Gamma_{v,\In},\\
        \|h(z)\| =& O(t^{\epsilon'})\text{ for }z\in\Gamma_{v,\Out}.
    \end{align*}
\end{lemma}

\begin{proof}
    Let us write $W_t|_{U_{\sigma_0}}$ as 
    \[
        W_t(z)=\sum_{v'\in V} h_{v'}(z)+\sum_{\alpha}h_\alpha(z),
    \]
    where $h_i(z) = t^{\lambda_i}z^{n_i}$, $h_{ij}(z)=t^{\lambda_{i-1}+\sum_{k=1}^jc_{ik}}z^{n_{i-1}+jn_i}$ and $\sum_{\alpha}h_\alpha(z) = (W_t-W_t^*)(z)$ with each $h_{\alpha}$ a monomial in $W_t-W_t^*$. Then for $z\in U_{\sigma_0}$, we have $\|h_{v'}(z)\| = t^{\beta_{v'}(\Log_t(z))}$ and $\|h_\alpha(z)\|=t^{\alpha(\Log_t(z))}$ with $\alpha(x)$ given in (\ref{eq:W_t_written_in_monomials}). 
    
    For each $h_{\alpha}$, there exists some $h_{v'}$ such that $\|h_{\alpha}(z)\| < t^{\epsilon'}\|h_{v'}(z)\|$ for $z\in U_{\sigma_0}$ due to the choice of $\epsilon'$ in Section \ref{sec:structuresheaf_decomposition}. So to prove the lemma, it is sufficient to show that $\|h_{v'}(z)\| = O(t^{\epsilon'})$ for $v'\in V$ and $z\in\Gamma_{v,\In}$, and $\|h_{v'}(z)h_i^{-1}(z)\| = O(t^{\epsilon'})$ for $v'\in V\backslash\{i\}$ and $z\in\Gamma_{v,\Out}$.
    
    Now for $z\in\Gamma_{v,\In}$, since $\Log_t(z)\in\Path_{v,\In}\subset\Polytope{\epsilon'}$, we have $\|h_{v'}(z)\| = t^{\beta_{v'}(\Log_t(z))} \leq t^{\epsilon'}$. For $z\in\Gamma_{v,\Out}$, since $\Log_t(z)\in \Path_{v,\Out}\subset R_i$, we have $\|h_{v'}(z)h_i^{-1}(z)\| = t^{\beta_{v'}(\Log_t(z)) - \beta_{i}(\Log_t(z))} < t^{\epsilon'}$.
\end{proof}
Using Lemma \ref{lemma:Gamma_O_L_W_approximation}, we have
\begin{align*}
    \int_{\Gamma_{v,\In}}e^{-W_t}\Omega =& \int_{\Gamma_{v,\In}}\left(1+(e^{-W_t}-1)\right)d\log(z^{n_i})\wedge d\log(z^{n'})\\
    =& (2\pi\sqrt{-1}\log t)a_v\int_{x\in\Path_{v,\In}} d\langle x,n_i\rangle + O(t^{\epsilon'}),\\
    \int_{\Gamma_{v,\Out}}e^{-W_t}\Omega =& \int_{\Gamma_{v,\Out}}e^{-t^{\lambda_i}z^{n_i}(1+h(z))}d\log(z^{n_i})\wedge d\log(z^{n'})\\
    =& (2\pi\sqrt{-1})a_v\int_{x\in\Path_{v,\Out}}e^{-t^{\beta_i(x)}(1+O(t^{\epsilon'}))} d\log(t^{\beta_i(x)}).
\end{align*}
where $n'\in N$ is any vector such that $n_i\wedge n'=1$, $a_v = \kappa_i$ for $v=i$ and $a_v = b_{ij}$ for $v=ij$. So it is sufficient to evaluate $\int_{x\in\Path_{v,\In}} d\langle x,n_i\rangle$ and $\int_{x\in\Path_{v,\Out}}e^{-t^{\beta_i(x)}(1+O(t^{\epsilon'}))} d\log(t^{\beta_i(x)})$ for $v=i$ or $ij$.

\begin{lemma}\label{lemma:Gamma_O_L_combinatoric_information}
    We have
    \begin{align*}
        \int_{x\in\Path_{i,\In}} d\langle x,n_i\rangle  =& \epsilon'-\lambda_i,\\
        \int_{x\in\Path_{ij,\In}} d\langle x,n_i\rangle =& \epsilon'+c_{ij}-\lambda,\\
        \int_{x\in\Path_{v,\Out}}e^{-t^{\beta_i(x)}(1+O(t^{\epsilon'}))} d\log t^{\beta_i(x)} =& -\gamma-\epsilon'\log t + O(t^\epsilon)\text{ for } v=i \text{ or } ij.
    \end{align*}
\end{lemma}

\begin{proof}
    Since $\Path_i(0)=0$ and $\Path_i(1)\in F_i$, we have $\langle \Path_i(0),n_i\rangle=0$ and $\langle \Path_i(1),n_i\rangle=\epsilon'-\lambda_i$. So 
    \begin{align*}
        \int_{x\in\Path_{i,\In}}d\langle x,n_i\rangle = \langle \Path_i(1)-\Path_i(0),n_i\rangle = \epsilon'-\lambda_i.
    \end{align*}
    Similarly, we have that 
    \begin{align*}
        \int_{x\in\Path_{ij,\In}} d\langle x,n_i\rangle =& \epsilon'+c_{ij}-\lambda.
    \end{align*}
    Now replacing $t^{\beta_i(x)}$ with $\delta$, we have
    \begin{align*}
    & \int_{x\in\Path_{v,\Out}}e^{-t^{\beta_i(x)}(1+O(t^{\epsilon'}))} d\log t^{\beta_i(x)}\\
    =& \int_{t^{\epsilon'}}^{\infty}e^{-\delta(1+O(t^{\epsilon'}))} d\log \delta\\
    =& \int_{t^{\epsilon'}}^{\infty} e^{-\delta}+\sum_{l=1}^{\infty}\frac{1}{l!}\left(e^{-\delta}(-\delta)^lO(t^{\epsilon'})^l\right)d\log \delta \\
    =& -\gamma-\epsilon'\log t+\sum_{l=1}^{\infty}\frac{1}{l}\left(-O(t^{\epsilon'})\right)^l\\
    =& -\gamma-\epsilon'\log t+O(t^\epsilon).
    \end{align*}
\end{proof}
So using Lemma \ref{lemma:Gamma_O_L_combinatoric_information}, we have that
\begin{align}
    \int_{\Gamma_i}e^{-W_t}\Omega =& (2\pi\sqrt{-1})\kappa_i(-\lambda_i\log t-\gamma) + O(t^\epsilon),\label{eq:integral_gamma_i}\\
    \int_{\Gamma_{ij}}e^{-W_t}\Omega =& (2\pi\sqrt{-1})b_{ij}\left((c_{ij}-\lambda_i)\log t-\gamma\right) + O(t^\epsilon).\label{eq:integral_gamma_ij}
\end{align}

For $z\in \Gamma_{\Center}$, since $\Log_t(z) = 0$ and $0\in \Polytope{\epsilon}$, $\|W_t(z)\| = O(t^\epsilon)$. So we have
\begin{equation}\label{eq:integral_gamma_center}
    \int_{\Gamma_{\Center}}e^{-W_t}\Omega = (2\pi\sqrt{-1})^2\int_{\sigma_L}(1+O(t^\epsilon))dy_1\wedge dy_2
    = (2\pi\sqrt{-1})^2\vol(\sigma_L).
\end{equation}

\subsection{Oscillatory integral on \texorpdfstring{$\Gamma_{\sing,ij,k}$}{}}

We firstly show that only the monomials $z^{ln_i},l\in\Z$ in the expansion $e^{-W_t} = \sum_{q=0}^{\infty}\frac{1}{q!}(-W_t)^q$ contribute to the oscillatory integral.

\begin{lemma}\label{lemma:Gamma_sing_ij_monomial_integral}
For $n\in N$ such that $n\neq ln_i$ for any $l\in\Z$,
    \begin{align*}
        \int_{\Gamma_{\sing,ij,k}}z^{n}\Omega = 0.
    \end{align*}
\end{lemma}

\begin{proof}
    Write $n$ as $n=an_i+bn'$ with $n'\in N$ such that $n_i\wedge n'=1$. Then we have
    \begin{align*}
        & \int_{\Gamma_{\sing,ij,k}}z^{n}\Omega\\
        =& \int_{\Gamma_{\sing,ij,k}}z^{n} d\log(z^{n_i})\wedge d\log(z^{n'})\\
        =& \int_{s\in[0,1]}\int_{\tau\in[0,1]}|t^{\langle \Path_{\sing,ij,k}(s),n \rangle}|e^{2\pi\sqrt{-1}b\tau} d\log(p_{k}(s))\wedge(2\pi\sqrt{-1}) d\tau.
    \end{align*}
    Since $\int_{\tau\in[0,1]}e^{2\pi\sqrt{-1}q\tau}d\tau=0$ whenever $q\in\Z_{\neq 0}$, $\int_{\Gamma_{\sing,ij,k}}z^{n}\Omega = 0$ when $b\neq 0$.
\end{proof}

We then show that the monomials $z^{ln_i}$ in $e^{-W_t}-1$ contribute small to the oscillatory integral.

\begin{lemma}\label{lemma:Gamma_sing_ij_monomial_small}
    Let $A$ be a real number such that
    \begin{equation*}
        0<A<\min_{\{q'\in\Z_{>0},l'\in\Z_{\neq 0}\}}\{|\frac{q'}{l'}| \mid t^{a'}z^{l'n_i} \text{ appears in }(W_t)^{q'}\}.
    \end{equation*}
    Then for any $q\in\Z_{>0}$, if $t^{a}z^{ln_i}$ is a monomial in $(W_t)^q$, then
    \begin{equation*}
        \|t^{a}z^{ln_i}\| = O(t^{q\epsilon})
    \end{equation*}
    for $z\in\Gamma_{\sing,ij,k}$.
\end{lemma}

\begin{proof}
    Suppose $t^{a}z^{ln_i} = \prod_{r=1}^q h_{r}(z)$ with $h_r$ a monomial in $W_t$. Then for any $z\in\Gamma_{\sing_{ij,k}}$, there exists some $s\in[0,1]$ such that $\|t^{a}z^{ln_i}\| = t^a\|p_k(s)\|^l$. Now due to the choice of $p_k(s)$ in Section \ref{sec:mirror_cycle_exceptional}, we have that
    \begin{align*}
        t^a\|p_k(s)\|^l =&(\|p_k(s)\|t^{c_{ij}})^lt^at^{-lc_{ij}}\\
        <& t^{A|l|(\epsilon-\epsilon')}t^at^{-lc_{ij}}\\
        <& t^{q(\epsilon-\epsilon')}t^at^{-lc_{ij}}.
    \end{align*}
    Note that $t^at^{-lc_{ij}} = t^a t^{l\langle x_{ij},n_i \rangle} = \prod_{r=1}^q h_{r}(t^{x_{ij}})$. Since $x_{ij}\in \Polytope{\epsilon'}$, $h_r(t^{x_{ij}})<t^{\epsilon'}$, so $t^a\|p_k(s)\|^l<t^{q\epsilon}$.
\end{proof}

Using Lemma \ref{lemma:Gamma_sing_ij_monomial_integral} and Lemma \ref{lemma:Gamma_sing_ij_monomial_small}, we have that
\begin{equation}\label{eq:integral_gamma_sing}
    \begin{split}
        \int_{\Gamma_{\sing,ij,k}}e^{-W_t}\Omega =& \int_{\Gamma_{\sing,ij,k}}(1+O(t^\epsilon))\Omega\\
        =& (1+O(t^\epsilon))\int_{s\in[0,1]}\int_{\tau\in[0,1]} d\log(p_{k}(s))\wedge(2\pi\sqrt{-1})d\tau\\
        =& (1+O(t^\epsilon))(2\pi\sqrt{-1})\int_{s\in[0,1]}d\log (p_{k}(s))\\
        =& (2\pi\sqrt{-1})\left((1-2k)\pi\sqrt{-1}\right)+O(t^\epsilon).
    \end{split}
\end{equation}

\subsection{Conclusion}
We need to compare $Z_{\topo}(\cO(L))$ with $Z_B(\Gamma_{\cO(L)})$. Since we have done it for the structure sheaf, it is sufficient to compare $Z_{\topo}(\cO(L)) - Z_{\topo}(\cO)$ with $Z_B(\Gamma_{\cO(L)})-Z_B(\Gamma_{\cO})$.
Say $L=L_{\toric}+L_{\Except}$ with $L$ toric and $L_{\Except}=\sum_{i=1}^n\sum_{j=1}^{l_i}b_{ij}E_{ij}$, then using (\ref{eq:Gamma_class}), we have that
\begin{align*}
    & Z_{\topo}(\cO(L)) - Z_{\topo}(\cO)\\
    =& (2\pi\sqrt{-1})\big(-(\omega\cdot L)\log t-\gamma c_1(Y)\cdot L\big) + \frac{1}{2}(2\pi\sqrt{-1})^2 L^2\\
    =& (2\pi\sqrt{-1})\sum\limits_{i=1}^n \big(-(\lambda_iD_i\cdot L_{\toric})\log t-\gamma D_i\cdot L_{\toric}\big) + \frac{1}{2}(2\pi\sqrt{-1})^2 L_{\toric}^2\\
    &+(2\pi\sqrt{-1})\sum\limits_{i=1}^n\sum_{j=1}^{l_i} b_{ij}\big(-(\omega\cdot E_{ij})\log t-\gamma c_1(Y)\cdot E_{ij}\big) - \frac{1}{2}(2\pi\sqrt{-1})^2 \sum\limits_{i=1}^n\sum_{j=1}^{l_i} b_{ij}^2.
\end{align*}
Since $D_i\cdot L_{\toric} = \kappa_i$, $\omega\cdot E_{ij} = \lambda_i-c_{ij}$, $c_1(Y)\cdot E_{ij} = 1$ and $\frac{1}{2}L_{\toric}^2 = \vol(\sigma_{L_{\toric}})$, we have that 
\begin{equation}
    Z_B(\Gamma_{\cO(L)})-Z_B(\Gamma_{\cO}) = Z_{\topo}(\cO(L)) - Z_{\topo}(\cO) +O(t^\epsilon)
\end{equation}
by using (\ref{eq:integral_gamma_i}), (\ref{eq:integral_gamma_ij}), (\ref{eq:integral_gamma_center}) and (\ref{eq:integral_gamma_sing}), which proves Conjecture \ref{conj:main} for $\cO(L)$.

\begin{theorem}
\label{thm:main}
There is an $\epsilon>0$, such that
    $$Z_B(E)=Z_\top(\alpha(E))+O(t^\epsilon),$$
    where $\alpha$ is given by \eqref{eq:mirror_cycle_of_a_general_divisor}.
\end{theorem}

\printbibliography

\end{document}